\documentclass{amsart}

\usepackage{amsfonts}
\usepackage{amssymb}
\usepackage{enumerate}
\usepackage{amsmath}
\usepackage{amsthm}
\usepackage{graphicx}
\usepackage[english]{babel}

\numberwithin{equation}{section}

\newtheorem{theorem}{Theorem}[section]
\newtheorem{lemma}[theorem]{Lemma}

\newtheorem{problem}[theorem]{Problem}

\newtheorem{corollary}[theorem]{Corollary}
\newtheorem{statement}[theorem]{Statement}
\newtheorem{fact}[theorem]{Fact}
\newtheorem*{m}{Theorem \ref{t:main}}
\newtheorem*{M}{Theorem \ref{t:Main}}
\newtheorem*{c:occup}{Corollary \ref{c:occup}}
\newtheorem*{c:MainHP}{Corollary \ref{c:MainHP}}
\newtheorem*{c:Mainp}{Corollary \ref{c:Mainp}}
\newtheorem*{c:HP}{Corollary \ref{c:HP}}
\newtheorem*{occup2}{Corollary \ref{c:occup2}}
\newtheorem*{ex}{Theorem \ref{t:ex}}
\newtheorem*{sing}{Theorem \ref{t:sing}}
\newtheorem*{nonshy}{Theorem \ref{t:nonshy}}
\newtheorem*{maxlevel}{Theorem \ref{t:maxlevel}}

\newtheorem*{t:graph}{Theorem \ref{t:graph}}
\newtheorem*{t:graph2}{Theorem \ref{t:graph2}}

\theoremstyle{definition}
\newtheorem{example}[theorem]{Example}
\newtheorem{remark}[theorem]{Remark}
\newtheorem{definition}[theorem]{Definition}
\newtheorem{notation}[theorem]{Notation}

\DeclareMathOperator{\inter}{int}

\DeclareMathOperator{\cl}{cl}
\DeclareMathOperator{\diam}{diam}
\DeclareMathOperator{\dist}{dist}
\DeclareMathOperator{\Lip}{Lip}

\DeclareMathOperator{\graph}{graph}
\DeclareMathOperator{\pr}{pr}
\DeclareMathOperator{\supp}{supp}
\DeclareMathOperator{\Var}{Var}
\DeclareMathOperator{\Prob}{Pr}
\DeclareMathOperator{\conv}{conv}

\newcommand{\NN}{\mathbb{N}}

\newcommand{\RR}{\mathbb{R}}

\newcommand{\iA}{\mathcal{A}}
\newcommand{\iB}{\mathcal{B}}
\newcommand{\iC}{\mathcal{C}}

\newcommand{\iG}{\mathcal{G}}
\newcommand{\iH}{\mathcal{H}}
\newcommand{\iI}{\mathcal{I}}
\newcommand{\iJ}{\mathcal{J}}
\newcommand{\iK}{\mathcal{K}}
\newcommand{\iL}{\mathcal{L}}

\newcommand{\iR}{\mathcal{R}}
\newcommand{\iF}{\mathcal{F}}
\newcommand{\iP}{\mathcal{P}}
\newcommand{\iS}{\mathcal{S}}
\newcommand{\iN}{\mathcal{N}}
\newcommand{\iU}{\mathcal{U}}
\newcommand{\iT}{\mathcal{T}}
\newcommand{\iV}{\mathcal{V}}

\newcommand{\iZ}{\mathcal{Z}}
\newcommand{\iY}{\mathcal{Y}}

\begin{document}

\title[Dimensions of fibers and graphs of prevalent continuous maps]{Hausdorff and packing dimension of fibers and graphs of prevalent continuous maps}
\author{Rich\'ard Balka}
\address{Current affiliation: Department of Mathematics, University of British Columbia, and Pacific Institute for the Mathematical Sciences, Vancouver,
BC V6T 1Z2, Canada}
\address{Former affiliations: Department of Mathematics, University of Washington, Box 354350, Seattle, WA 98195-4350, USA and Alfr\'ed R\'enyi Institute of Mathematics, Hungarian Academy of Sciences, PO Box 127, 1364 Budapest, Hungary}
\email{balka@math.ubc.ca}
\thanks{The first author was supported by the
Hungarian Scientific Research Fund grants no.~72655 and 104178. The third author was supported by the Hungarian
Scientific Research Fund grants no.~72655, 83726 and 104178.}

\author{Udayan B. Darji}

\address{Department of Mathematics, University of Louisville, Louisville, KY 40292,
USA}

\email{ubdarj01@louisville.edu}

\author{M\'arton Elekes}

\address{Alfr\'ed R\'enyi Institute of Mathematics, Hungarian Academy of Sciences,
PO Box 127, 1364 Budapest, Hungary and E\"otv\"os Lor\'and
University, Institute of Mathematics, P\'azm\'any P\'eter s. 1/c,
1117 Budapest, Hungary}

\email{elekes.marton@renyi.mta.hu}

\subjclass[2010]{Primary: 28A78, 28C10, 46E15, 60B05; Secondary: 54E52.}

\keywords{Haar null, shy, prevalent, Hausdorff dimension, packing dimension, level set, fiber, graph, continuous map, H\"older map, Lipschitz map, ultrametric space, Baire category, generic, occupation measure, Brownian motion.}

\begin{abstract} The notions of shyness and prevalence generalize the property of being zero and full Haar measure to arbitrary (not necessarily locally compact) Polish groups.
The main goal of the paper is to answer the following question:
What can we say about the Hausdorff and packing dimension of the fibers of prevalent continuous maps?

Let $K$ be an uncountable compact metric space. We prove that the prevalent $f\in C(K,\RR^d)$
has many fibers with almost maximal Hausdorff dimension.
This generalizes a theorem of Dougherty and yields that the prevalent $f\in C(K,\RR^d)$ has graph of maximal Hausdorff dimension, generalizing a result of Bayart and Heurteaux. We obtain similar results for the packing dimension.

We show that for the prevalent $f\in C([0,1]^m,\RR^d)$ the set of $y\in f([0,1]^m)$ for which $\dim_H f^{-1}(y)=m$ contains a
dense open set having full measure with respect to the occupation measure $\lambda^m \circ f^{-1}$, where $\dim_H$ and
$\lambda^m$ denote the Hausdorff dimension and the $m$-dimensional Lebesgue measure, respectively. We also prove an analogous result when $[0,1]^m$ is replaced by any self-similar set satisfying the open set condition.

We cannot replace the occupation measure with Lebesgue measure in the above statement: We show that the functions
$f\in C[0,1]$ for which positively many level sets are singletons form a non-shy set in $C[0,1]$.
In order to do so, we generalize a theorem of Antunovi\'c, Burdzy, Peres and Ruscher. As a complementary result
we prove that the functions $f\in C[0,1]$ for which
$\dim_H f^{-1}(y)=1$ for all $y\in (\min f,\max f)$ form a non-shy set in $C[0,1]$.

We also prove sharper results in which large Hausdorff dimension is replaced by positive measure with respect to generalized Hausdorff measures, which answers a problem of Fraser and Hyde.
\end{abstract}

\maketitle

\tableofcontents

\section{Introduction}

Let $G$ be a \emph{Polish group}, that is, a separable topological group which is endowed with a compatible complete metric.
If $G$ is locally compact then there exists a \emph{Haar measure} on $G$, that is,
a left translation invariant regular Borel measure which is finite on compact sets and positive on non-empty open sets.
The concept of Haar measure does not extend to groups that are not locally compact, but the idea of Haar measure zero sets does. The following definition is due to Christensen \cite{C} and was rediscovered by Hunt, Sauer and York \cite{HSY}.

\begin{definition} \label{d:shy} For an abelian Polish group $G$ a set $A\subset G$ is \emph{shy} or \emph{Haar null}
if there exists a Borel set $B\subset G$ and a Borel probability measure $\mu$ on $G$ such that $A\subset B$ and $\mu\left(B+x\right) = 0$ for all $x\in G$.
The complement of a shy set is called a \emph{prevalent} set.
\end{definition}

Christensen proved in \cite{C} that shy sets form a $\sigma$-ideal and in locally compact abelian Polish groups Haar measure zero sets and shy sets coincide.
Later Tops{\o}e and Hoffmann-J{\o}rgensen \cite{THJ} and Mycielski \cite{M} extended the definition to all Polish groups,
but here we consider only the abelian case.

\begin{notation} The Hausdorff and packing dimension of a metric space $X$ is denoted by $\dim_H X$ and $\dim_P X$.
We use the convention $\dim_H \emptyset=\dim_P  \emptyset=-1$.
For a compact metric space $K$ let us denote by $C(K,\mathbb{R}^d)$ the set of continuous functions from $K$ to $\mathbb{R}^d$
endowed with the maximum norm. Then $C(K,\mathbb{R}^d)$ is a Banach space. We simply write $C[0,1]=C([0,1],\RR)$.
\end{notation}

Over the last 25 years there has been a large interest in studying dimensions of various sets related to `typical' continuous functions. If typical means generic in the sense of Baire category, then the following theorem about level sets is folklore.

\begin{theorem}\label{t:folklore} For the generic $f\in C[0,1]$ for all $y\in f([0,1])$ we have
$$\dim_H f^{-1}(y)=0.$$
\end{theorem}

Mauldin and Williams \cite{MW} proved the next theorem.

\begin{theorem}[Mauldin-Williams] \label{t:MW} For the generic $f\in C[0,1]$ we have
$$\dim_H \graph(f)=1.$$
\end{theorem}

As for the higher dimensional analogues, the next result was obtained by Kirchheim \cite{Ki}.

\begin{theorem}[Kirchheim]
If $m,d\in \NN^+$ and  $m\geq d$ then for the generic
$f \in C([0,1]^m, \RR^d)$ for all $y\in \inter f\left([0,1]^m\right)$ we have
\[
\dim_H f^{-1}(y)=m-d.
\]
\end{theorem}

Now let $K$ be an arbitrary compact metric space. In order to determine the Hausdorff dimension of the level sets of the generic $f\in C(K,\RR)$,
we need a new notion of dimension, the \emph{topological Hausdorff dimension}, see \cite{BBE} and \cite{BBE2}.
More generally, the right concept to describe the Hausdorff dimension of the fibers of the generic $f\in C(K,\RR^d)$ is the so-called \emph{$d^\textrm{th}$ inductive topological Hausdorff dimension}, see \cite{B}.

The case of graphs is much simpler, the strategy of Mauldin and Williams actually easily yields the following general result, see also \cite{BBE2}.

\begin{theorem}
Let $K$ be an uncountable compact metric space and $d\in \mathbb{N}^+$.
Then for the generic $f\in C(K, \RR^d)$ we have
\[
\dim_H \graph(f)=\dim_H K.
\]
\end{theorem}

These theorems indicate that the generic $f\in C[0,1]$ behaves quite regularly in
a sense, e.g. its level sets and graph have minimal Hausdorff
dimension, similarly to the case of smooth functions. It is quite natural to expect more chaotic
behavior from typical continuous functions, which is already a reason to replace genericity with another notion.
Moreover, since these problems are measure theoretic in nature, it is natural to
replace Baire category by the more measure theoretic concept of prevalence.

\bigskip

In contrast to Theorem~\ref{t:folklore}, we show that the prevalent $f\in C[0,1]$ has fibers of maximal Hausdorff dimension.
Let us denote by $\lambda$ the one-dimensional Lebesgue measure.
(Note that $\dim_H X\leq \dim_P X$ for every metric space $X$, so the packing dimension analogue of the following
statement would be weaker.)

\begin{c:occup} For the prevalent $f\in C[0,1]$ there is an open set $U_f\subset \RR$ such that $\lambda(f^{-1}(U_f))=1$ (hence $U_f$
is dense in $f([0,1])$) and for all $y\in U_f$ we have
$$\dim_H f^{-1}(y)=1.$$
\end{c:occup}

In general, prevalent continuous maps have many fibers of cardinality continuum, for the following theorem see
\cite[Theorem~11]{D} and the remark following its proof.

\begin{theorem}[Dougherty] \label{t:D} Let $K$ be an uncountable compact metric space\footnote{Dougherty proved this result only if $K$ is the triadic Cantor set. As every uncountable compact metric space contains a subset homeomorphic to the triadic Cantor set by \cite[Cor.~6.5]{K}, Corollary~\ref{c:hereditary} yields this more general result.} and let $d\in \mathbb{N}^+$.
Then for the prevalent $f\in C(K,\mathbb{R}^d)$ we have
$$\inter f(K)\neq \emptyset.$$
Moreover, there is a non-empty open set $U_f\subset \mathbb{R}^d$ such that for all $y\in U_f$ we have
$$\# f^{-1}(y)=2^{\aleph_0}.$$
\end{theorem}

The next theorem widely generalizes Corollary~\ref{c:occup} and Theorem~\ref{t:D} in Euclidean spaces.
We can find many fibers not only of cardinality continuum, but also with almost maximal Hausdorff and packing dimension.
A Borel measure $\mu$ on a metric space $X$ is called a \emph{mass distribution} if $0<\mu(X)<\infty$.
For the definition of dimensions of mass distributions and their properties see the Preliminaries section.

\begin{M}[Main Theorem, simplified version] Assume that $m,d\in \NN^+$ and $K\subset \RR^m$ is compact.
If $\mu$ is a continuous mass distribution on $K$, then for the prevalent $f\in C(K,\RR^d)$ there is an open set $U_f\subset \RR^d$ such that $\mu(f^{-1}(U_f))=\mu(K)$ and for all $y\in U_f$ we have
$$\dim_H f^{-1}(y)\geq \dim_H \mu \quad \textrm{and} \quad \dim_P f^{-1}(y)\geq \dim_P \mu.$$
\end{M}

After some technical lemmas in Section~\ref{s:tech}, we prove the above theorem for
$K\subset \RR$ and $\mu=\lambda$ in Subsection~\ref{ss:real}, which is the most subtle proof of the paper.
In Subsection~\ref{ss:ultra} we prove this result for ultrametric spaces using ideas from \cite{KMZ}.
In Subsection \ref{ss:main} we finish the proof of the Main Theorem,
we trace back the case of general compact spaces to ultrametric ones by using a theorem of Zindulka \cite{Z}.
Let us denote by $\lambda^m$ the $m$-dimensional Lebesgue measure.

\begin{occup2} Let $m,d\in \NN^+$.
Then for the prevalent $f\in C([0,1]^m,\RR^{d})$
there is an open set $U_{f}\subset \RR^d$ such that
$\lambda^m(f^{-1}(U_f))=1$ (hence $U_f$ is dense in $f([0,1]^m)$ and for all $y\in U_{f}$ we have
$$\dim_{H} f^{-1} (y)=m.$$
\end{occup2}

\begin{c:Mainp}[simplified version]
Let $m,d\in \NN^+$ and let $K\subset \RR^m$ be an uncountable compact set. Then for the prevalent
$f\in C(K,\RR^d)$ for all $s<\dim_P K$ there is a non-empty open set $U_{f,s}\subset \RR^d$ such that for all $y\in U_{f,s}$
we have
$$\dim_P f^{-1} (y)\geq s.$$
In particular, we have
$$\sup\{\dim_P f^{-1}(y): y\in \RR^d\}=\dim_P K.$$
\end{c:Mainp}

In the case of Hausdorff dimension we prove more general versions of the above two corollaries
based on a deep theorem of Mendel and Naor \cite{MN}.

\begin{m}
Let $K$ be an uncountable compact metric space and let $d\in \mathbb{N}^+$. Then for the prevalent
$f\in C(K,\mathbb{R}^d)$ for all $s<\dim_H K$ there is a non-empty open set $U_{f,s}\subset \mathbb{R}^d$ such that for all $y\in U_{f,s}$
we have
$$\dim_H f^{-1}(y)\geq s.$$
In particular, we have
$$\sup\{\dim_H f^{-1}(y): y\in \mathbb{R}^d\}=\dim_H K.$$
\end{m}

The supremum is not necessarily attained in the second claims of Corollary~\ref{c:Mainp} and Theorem~\ref{t:main}.

\begin{ex} There is a compact set $K\subset \RR$ such that $\dim_H K = \dim_P K = 1$ and
$$\{f\in C(K,\RR): \dim_H f^{-1}(y) \le \dim_P f^{-1}(y)<1 \textrm{ for all } y\in \RR\}$$
is non-shy in $C(K,\RR)$.
\end{ex}

If $K$ is `large in its dimension' then the Main Theorem implies the following.

\begin{c:MainHP}[simplified version] Let $m,d\in \RR^d$, and let $K\subset \RR^m$ be compact. Let
$\dim$ be one of $\dim_H$ or $\dim_P$. Assume that there is a continuous mass distribution $\mu$ on $K$ such that $\supp \mu=K$.
Then for the prevalent $f\in C(K,\RR^d)$ there is an open set $U_f\subset \RR^d$ such that $\mu(f^{-1}(U_f))=\mu(K)$ and for all $y\in U_f$
we have
$$\dim f^{-1}(y)=\dim K.$$
\end{c:MainHP}
For sufficiently homogeneous spaces we can generalize the Main Theorem.
Let us denote by $\iH^s$ and $\iP^s$ the $s$-dimensional Hausdorff and packing measure, respectively. For the definitions of
packing measure, self-similar set, and open set condition see \cite{F2}.
\begin{c:HP} Let $m,d\in \NN^+$ and let $K\subset \RR^m$ be a self-similar set satisfying the open set
condition. It is well-known that $\dim_H K=\dim_P K=s$ and $\iH^s(K),\iP^s(K)\in \RR^+$.
Then for the prevalent $f\in C(K,\RR^d)$ there exists an open set $U_f\subset \RR^d$ such that
$\iH^s(f^{-1}(U_f))=\iH^s(K)$ (hence $U_f$ is dense in $f(K)$) and
$$\dim_H f^{-1}(y)=s \textrm{ for all } y\in U_f.$$
Similarly, for the prevalent $f\in C(K,\RR^d)$ there exists an open set $V_f\subset \RR^d$ such that
$\iP^s(f^{-1}(V_f))=\iP^s(K)$ (hence $V_f$ is dense in $f(K)$) and
$$\dim_P f^{-1}(y)=s \textrm{ for all } y\in V_f.$$
\end{c:HP}

For other results in sufficiently homogeneous spaces see Subsection~\ref{ss:hom},
where we describe the compact metric spaces $K$ for which
$\dim_H f^{-1}(y)=\dim_H K$ for the prevalent $f\in C(K,\RR^d)$ and the generic $y\in f(K)$. The characterization is independent of $d$.

Corollary~\ref{c:occup} yields for the prevalent $f\in C[0,1]$ that
$\{y: \dim_H f^{-1}(y)=1\}$ is co-meager in $f([0,1])$
with full $\lambda\circ f^{-1}$ measure. As the main result of Section~\ref{s:sing}, we show that this does not remain true if
we replace the occupation measure $\lambda\circ f^{-1}$ by the Lebesgue measure $\lambda$ on $f([0,1])$.
Let $\exists^{\lambda}$ denote that there exists positively many with respect to $\lambda$.

\begin{nonshy} The set
$$\{f\in C[0,1]:  \exists^{\lambda} y\in \RR \textrm{ such that } f^{-1}(y) \textrm{ is a singleton}\}$$
is non-shy in $C[0,1]$.
\end{nonshy}

Let $\iZ(f)=\{x\in [0,1]: f(x)=0\}$, the next theorem is \cite[Proposition~3.3]{ABPR}.

\begin{theorem}[Antunovi\'c-Burdzy-Peres-Ruscher] \label{t:ABPR}
Let $\mu$ be the Wiener measure on $C[0,1]$. Then there exists a function $g\in C[0,1]$ such that
$$\mu(\{f\in C[0,1]: \iZ(f-g)\setminus \{0\} \textrm{ is a singleton}\})>0.\footnote{It is easy to see that if $g(x)=x^{1/3}$ then $\iZ(f-g)=\{0\}$ for positively many $f$ with respect to the Wiener measure. In order to avoid this degenerate case, we remove the origin.}$$
\end{theorem}

The next theorem generalizes Theorem~\ref{t:ABPR} and easily implies Theorem~\ref{t:nonshy}.

\begin{sing} Let $\mu$ be a Borel probability measure on $C[0,1]$. Then there exists a function $g\in C[0,1]$ such that
$$\mu(\{f\in C[0,1]: \iZ(f-g) \textrm{ is a singleton}\})>0.$$
Consequently, the set $\{f\in C[0,1]: \iZ(f) \textrm{ is a singleton}\}$ is non-shy.
\end{sing}

As a complement to Theorem~\ref{t:nonshy}, we prove that all non-extremal level sets can be large. The goal of
Section~\ref{s:maxlevel} is to prove the following.

\begin{maxlevel}  The set
$$\{f\in C[0,1]: \dim_H f^{-1}(y)=1 \textrm{ for all } y\in  (\min f, \max f)\}$$
is non-shy in $C[0,1]$.
\end{maxlevel}

Recently, describing the various fractal dimensions of graphs of prevalent
continuous functions has attracted notable attention, this is the topic of Section~\ref{s:graph}.

First McClure \cite{Mc} proved that the packing dimension, and thus the upper box dimension of the graph of the prevalent $f\in C[0,1]$ is 2.
The analogous result for the lower box dimension
was proved in \cite{FF}, \cite{GJMNOP}, and \cite{Sh}, independently.

Fraser and Hyde \cite{FH} generalized the above results by showing that the prevalent
$f\in C[0,1]$ has graph of Hausdorff dimension 2.
In contrast to Theorem~\ref{t:MW} this means that the prevalent value of $\dim_H \graph(f)$ is as large as possible.

\begin{theorem}[Fraser-Hyde] \label{t:FH} For the prevalent $f\in C[0,1]$ we have
$$\dim_H \graph(f)=2.$$
\end{theorem}

The next result was proved by Bayart and Heurteaux, see \cite[Theorem~3]{BH}.

\begin{theorem}[Bayart-Heurteaux] \label{t:BH} If $K\subset \RR^m$ is compact with $\dim_H K>0$ then for the prevalent $f\in C(K,\RR)$
we have
$$\dim_H \graph(f)=\dim_H K+1.$$
\end{theorem}

The proof of Theorem~\ref{t:BH} is based on potential theoretic methods, they give a lower estimate for the Hausdorff dimension of $\graph(X+f)$,
where $X\colon K\to \RR$ is a fractional Brownian motion restricted to $K$ and $f\in C(K,\RR)$ is a continuous drift. Note that if $X\colon K\to \RR^d$ is a fractional Brownian motion restricted to some $K\subset [0,1]$ and $f\in C(K,\RR^d)$, then Peres and Sousi \cite{PS} determined the almost sure Hausdorff dimension of $\graph(X+f)$ in terms of $f$ and the Hurst index of $X$. It is not difficult to extend the proof of \cite[Theorem~3]{BH} to vector valued functions, and Theorem~\ref{t:D} handles the case $\dim_H K=0$. These yield the following theorem.

\begin{theorem}\label{t:gr} Let $m,d\in \NN^+$ and let $K\subset \RR^m$ be an uncountable compact set. Then for the prevalent $f\in C(K,\RR^d)$
we have
$$\dim_H \graph(f)=\dim_H K+d.$$
\end{theorem}

We will show that Theorem~\ref{t:main} also easily implies the above theorem. Moreover, the condition $K\subset \RR^m$ is superfluous.

\begin{t:graph} Let $K$ be an uncountable compact metric space and let $d\in \NN^+$. Then for the prevalent $f\in C(K,\RR^d)$
we have
$$\dim_H \graph(f)=\dim_H K+d.$$
\end{t:graph}

Much less was known about the prevalent value of the packing dimension of the
graphs. Corollary~\ref{c:Mainp} implies the packing dimension analogue of Theorem~\ref{t:gr}.

\begin{t:graph2}[simplified version] Let $m,d\in \NN^+$ and let $K\subset \RR^m$ be an uncountable compact set.
Then for the prevalent $f\in C(K,\RR^d)$ we have
$$\dim_P \graph(f)=\dim_P K+d.$$
\end{t:graph2}

In Section~\ref{s:gauge} we indicate how to obtain stronger forms of the main
results by replacing large dimension by positive measure with respect to generalized Hausdorff measures.
Finally, in Section~\ref{s:open} we pose some open problems.

\section{Preliminaries}

Let $(X,d)$ be a metric space.
For $A,B \subset X$ let us define $\dist(A,B) = \inf\{d(x,y) : x\in A,~y\in B\}$.
Let $B(x,r)$ and $U(x,r)$ be the closed and open ball of
radius $r$ centered at $x$, respectively. Set $B(A,r) = \{x \in X : \dist(\{x\},A) \le r \}$.
We denote by $\cl A$, $\inter A$ and $\partial A$ the closure, interior and boundary of $A$, respectively.
The diameter of a set $A$ is denoted by $\diam A$.
We use the conventions $\diam \emptyset = 0$ and $\inf \emptyset=\infty$.
For two metric spaces $(X,d_{X})$ and $(Y,d_{Y})$ a map
$f\colon X\to Y$ is \emph{$s$-H\"older} for an $s>0$
if there is a constant $c\in \mathbb{R}$ such that $d_{Y}(f(x_{1}),f(x_{2}))\leq c(d_{X}(x_{1},x_{2}))^{s}$
for all $x_{1},x_{2}\in X$. A map $f\colon X\to Y$ is \emph{Lipschitz} if it is $1$-H\"older, and the
smallest $c$ in the definition is called the Lipschitz constant of $f$ and is denoted by
$\Lip(f)$. We say that $f$ is \emph{bi-Lipschitz}
if it is one-to-one and both $f$ and $f^{-1}$ are Lipschitz.

Let $s \ge 0$. The \emph{$s$-dimensional Hausdorff measure} of a metric space $X$ is
\begin{align*}
\mathcal{H}^{s}(X)&=\lim_{\delta\to 0+}\mathcal{H}^{s}_{\delta}(X)
\mbox{, where}\\
\mathcal{H}^{s}_{\delta}(X)&=\inf \left\{ \sum_{i=1}^\infty (\diam
X_{i})^{s}: X \subset \bigcup_{i=1}^{\infty} X_{i},~
\forall i \diam X_i \le \delta \right\}.
\end{align*}
Let $\dim_H \emptyset=-1$. The \emph{Hausdorff dimension} of a non-empty $X$ is defined as
\[
\dim_{H} X = \inf\{s \ge 0: \mathcal{H}^{s}(X) =0\},
\]
for more information on these concepts see \cite{F} or \cite{Ma}. Now we define the packing dimension.
If $X$ is non-empty and totally bounded then for all
$\delta>0$ let $N_{\delta}(X)$ be the smallest number of closed balls of radius $\delta$ whose union cover $X$.
Then the \emph{upper box dimension} of $X$ is defined as
$$\overline{\dim}_{B} X=\limsup_{\delta \to 0+} \frac{\log N_{\delta}(X)}{\log (1/\delta)}.$$
Let $\overline{\dim}_{B} \emptyset =-1$ and let $\overline{\dim}_{B} X=\infty$ if $X$ is not totally bounded.
The \emph{packing dimension} of $X$ is defined as
$$\dim_P X=\inf \left\{\sup_{i} \overline{\dim}_{B} X_i: X=\bigcup_{i=1}^{\infty} X_i\right\}.$$
Then clearly $\dim_P \emptyset =-1$. Since we do not need the packing measure, it was more convenient for us to define
the packing dimension as the modified upper box dimension, see e.g.\ \cite{F} or \cite{MP} for more on these concepts.
The following fact is an easy consequence of the definitions.

\begin{fact}\label{f:Holder} If $X,Y$ are non-empty metric spaces and $f\colon X\to Y$ is $s$-H\"older then
$$\dim_H f(X)\leq \frac{\dim_H X}{s} \quad \textrm{and} \quad \dim_P f(X)\leq \frac{\dim_P X}{s}.$$
\end{fact}

Let $K$ be a compact metric space an let $\mu$ be a mass distribution on $K$.
We define
\begin{align*} \dim_H \mu&=\inf \{\dim_H B: B\subset K \textrm{ is Borel and } \mu(B)>0\},\\
\dim_P \mu&=\inf \{\dim_P B: B\subset K \textrm{ is Borel and } \mu(B)>0\}.
\end{align*}

For the following theorem see \cite[Proposition~10.2]{F2} when $K$ is a subset of a Euclidean space.
In fact, the proof of \cite[Proposition~2.2]{F2}
with the covering theorem \cite[Theorem~2.1]{Ma} works in an arbitrary compact metric space.

\begin{theorem}\label{t:limsup} If $\mu$ is a mass distribution on a compact metric space $K$ then
\begin{align*} \dim_H \mu&=\sup \left\{s\geq 0: \limsup_{r\to 0+} \frac{\mu(B(x,r))}{r^{s}}<\infty \textrm{ for $\mu$-a.e. }x\in K\right\},\\
 \dim_P \mu&=\sup \left\{s\geq 0: \liminf_{r\to 0+} \frac{\mu(B(x,r))}{r^{s}}<\infty \textrm{ for $\mu$-a.e. }x\in K\right\}.
\end{align*}
\end{theorem}

The next theorem states that we can approximate the dimension of a compact metric space $K$
by the dimension of measures supported within it. For the proof see the theorem above with Frostman's lemma \cite[Theorem~8.17]{Ma} and \cite{IT} in
the case of the Hausdorff and the packing dimension, respectively.
Moreover, we may assume that the measures are Hausdorff and packing measures restricted to a compact subset of $K$, see
\cite{Ho} and \cite{JP}, respectively. See also \cite[Proposition~10.1]{F2} for the Euclidean case.

\begin{theorem}\label{t:frostman} If $K$ is a non-empty compact metric space then
\begin{align*} \dim_H K&=\sup \{\dim_H \mu: \mu \textrm{ is a mass distribution on } K\}, \\
\dim_P K&=\sup \{\dim_P \mu: \mu \textrm{ is a mass distribution on } K\}.
\end{align*}
If $K$ is uncountable then we may assume that the above measures $\mu$ are continuous.
\end{theorem}

The metric space $(X,d)$ is called \emph{ultrametric} if the triangle inequality is replaced with the stronger inequality
$d(x,y)\leq \max\{d(x,z),d(y,z)\}$ for all $x,y,z\in X$.
\begin{fact} \label{f:ultra} Let $X$ be an ultrametric space.
Then for all $x,y\in X$ and $r>0$ either $B(x,r)\cap B(y,r)=\emptyset$ or $B(x,r)=B(y,r)$.
\end{fact}

Let $X$ be a \emph{complete} metric space. A set is \emph{somewhere dense} if
it is dense in a non-empty open set, and otherwise it is called \emph{nowhere dense}. We say that $A \subset X$ is
\emph{meager} if it is a countable union of nowhere dense sets, and
a set is called \emph{co-meager} if its complement is meager. By
Baire's category theorem a set is co-meager iff it contains a dense
$G_\delta$ set. We say that the \emph{generic} element $x \in X$ has
property $\mathcal{P}$ if $\{x \in X : x \textrm{ has property }
\mathcal{P} \}$ is co-meager. Our main example will be $X=C(K,\RR^d)$. See e.g. \cite{K} for
more on these concepts.

A metric space $X$ is a \emph{Polish space} if it is complete and separable. We say that $A\subset X$
\emph{analytic} if it is a continuous image of a Polish space, and \emph{co-analytic}
if its complement is analytic. A Borel subset of a Polish space is analytic, see
\cite[Theorem~13.7]{K}. Continuous images, countable unions and countable intersections of analytic sets are also analytic
\cite[Proposition~14.4]{K}. For more on these concepts see \cite{K}.

Let $\mu$ be a mass distribution on a Polish space $X$.
Then $\mu$ can be extended to the $\sigma$-algebra of the $\mu$-measurable sets as a complete measure, see \cite[113C]{Fr}.
Analytic and co-analytic sets are $\mu$-measurable \cite[434D~(c)]{Fr}. We denote by $\supp \mu$ the \emph{support of $\mu$},
the minimal closed subset $F$ of $X$ so that $\mu(X\setminus F) =0$. The measure $\mu$ is called \emph{continuous} is $\mu(\{x\})=0$ for all $x\in X$.
For the following classical theorems see \cite[433C]{Fr} and \cite[Theorem~A,~p.~54.]{Ha}, respectively.

\begin{theorem}\label{t:Ulam} If $X$ is a Polish space and $\mu$ is a mass distribution on $X$
then there is a compact set $K\subset X$ with $\mu(K)>0$.
\end{theorem}

\begin{theorem}[Carath\'eodory's extension theorem] Any $\sigma$-finite measure defined on an algebra $\iA$
can be \emph{uniquely} extended to the $\sigma$-algebra generated by $\iA$.
\end{theorem}

Let $G$ be an abelian Polish group and let $\mu$, $\nu$ be $\sigma$-finite Borel measures on $G$. For a Borel set $A\subset G$ let us define
$$(\mu \ast \nu)(A)= (\mu \times \nu)(\{(x,y)\in G\times G: x+y\in A\}),$$
where $\mu\times \nu$ is the product measure on $G\times G$. Then $\mu\ast \nu$ is a $\sigma$-finite Borel measure on $G$ called the
\emph{convolution of $\mu$ and $\nu$}.

For all $f\in C[0,1]$ let $\iZ(f)=\{x\in [0,1]: f(x)=0\}$. If $x=(x_1,\dots,x_d)\in \RR^d$ then the
\emph{maximum norm} of $x$ is defined as $||x||=\max_{1\leq i\leq d}|x_i|$.
Let $\chi_{A}$ be the characteristic function of the set $A$.
If $A\subset \RR$ then let $\conv (A)$ be the convex hull of $A$.
We denote by $\Pr$, $ \mathbb{E} $ and $\Var$ the probability, expected value and variance, respectively.

\section{Technical lemmas}\label{s:tech}

Our definition of prevalence follows Hunt, Sauer and York \cite{HSY} and differs from Christensen \cite{C}
in which the definition is given for so-called universally measurable sets (without the Borel hulls).
These definitions are equivalent for Borel sets, but they differ in general, see \cite{EV}.
The following theorem states that the definitions are also equivalent for co-analytic sets,
see \cite[Proposition~(i)]{S} for the proof.

\begin{theorem}[Solecki] \label{t:S}
Let $G$ be an abelian Polish group and let $A\subset G$ be a co-analytic set. If there exists a Borel probability measure
$\mu$ on $G$ such that $\mu(A+g)=1$ for all $g\in G$ then $A$ is prevalent.
\end{theorem}

The following lemma is basically \cite[Lemma~2.11]{BBE2}. It is only stated there in the special case $d=1$,
but the proof works verbatim for all $d\in \NN^+$.

\begin{lemma} \label{l:B} Let $K$ be a compact metric space, let $d\in \NN^+$ and $c\in \RR$. Then
$$\Delta=\left\{(f,y)\in C(K,\RR^d)\times \RR^d: \dim_H f^{-1}(y)<c\right\}$$
is a Borel set in $C(K,\RR^d)\times \RR^d$.
\end{lemma}

\begin{lemma} \label{l:co} Let $K\subset \RR$ be compact, let $d\in \NN^+$ and $c\in \RR$. Then
\begin{align*} \iA=\{&f\in C(K,\RR^d): \exists \textrm{ an open set } U_f\subset \RR^d \textrm{ such that}\\
&\lambda\left(f^{-1}(U_f)\right)=\lambda (K) \textrm{ and } \dim_H f^{-1}(y)\geq c \textrm{ for all } y\in U_f\}
\end{align*}
is co-analytic in $C(K,\RR^d)$.
\end{lemma}

\begin{proof} Let $\iV$ be a countable basis of $\RR^d$ and let $\iU$ be the family of finite unions
of elements of $\iV$. Clearly $\iU$ is countable and $\iA=\bigcup_{U\in \iU} \bigcap_{n=1}^{\infty} \iA_{n,U}$, where
\begin{align*} \iA_{n,U}=\{&f\in C(K,\RR^d): \lambda(f^{-1}(U)) > \lambda (K)-1/n \\
&\textrm{and }  \dim_H f^{-1}(y)\geq c \textrm{ for all } y\in U\}.
\end{align*}
As co-analytic sets are closed under countable union and countable intersection, it is enough to prove that the $\iA_{n,U}$ are co-analytic.
Fix $n\in \NN^+$ and $U\in \iU$ and let
\begin{align*} \iB&=\{f\in C(K,\RR^d): \dim_H f^{-1}(y)\geq c \textrm{ for all } y\in U\}, \\
\iC&=\{f\in C(K,\RR^d): \lambda\left(f^{-1}(U)\right) > \lambda (K)-1/n\}.
\end{align*}
Since $\iA_{n,U}=\iB\cap \iC$, it is enough to prove that
$\iB$ and $\iC$ are co-analytic.

First we show that $\iB$ is co-analytic. By Lemma~\ref{l:B} the set
$$\Delta=\left\{(f,y)\in C(K,\RR^d)\times \RR^d: \dim_H f^{-1}(y)<c\right\}$$
is Borel. Define $\pr \colon C(K,\RR^d)\times \RR^d \to C(K,\RR^d)$ as $\pr(f,y)=f$. Then
$$\iB=\left( \pr\left(\Delta \cap (C(K,\RR^d)\times U)\right)\right)^c$$
is the complement of the projection of a Borel set. Hence $\iB$ is co-analytic.

Finally, we prove that $\iC$ is Borel. For all $r\in \RR$ let
$$\iC(r)=\{f\in C(K,\RR^d): \lambda\left(f^{-1}(U)\right)>r\}.$$
It is enough to prove that the $\iC(r)$ are open. Fix $r\in \RR$ and assume that
$f\in \iC(r)$, that is, $\lambda(f^{-1}(U))>r$. We need to find an $\varepsilon>0$ such that $U(f,\varepsilon)\subset \iC(r)$.
The regularity of the Lebesgue measure implies that there is a compact set $C\subset f^{-1}(U)$ with
$\lambda(C)>r$. As $f(C)\subset U$ is compact, we can define $\varepsilon=\dist(f(C),\RR^d\setminus U)>0$.
Clearly $g(C)\subset U$ for every $g\in U(f,\varepsilon)$, thus $\lambda(g^{-1}(U))\geq \lambda(C)>r$.
Hence $U(f,\varepsilon)\subset \iC(r)$, and the proof is complete.
\end{proof}

\begin{definition} \label{d:Cantor1} Let $\{ a_n \}_{n\in \NN^+}$ be a sequence of positive integers.
A compact set $K\subset \RR$ is an \emph{$(a_n)$-type fat Cantor set} if $\lambda(K)>0$ and it is of the form
\begin{equation} \label{eq:defK} K=\bigcap _{n=1}^{\infty}\left(\bigcup_{i_1=1}^{a_1}\dots \bigcup_{i_n=1}^{a_n} K_{i_1 \dots i_n} \right),
\end{equation}
where $K_{i_1\dots i_n}\subset K$ are compact sets such that for every
$n\in \mathbb{N}^+$ and for each distinct $(i_1,\dots ,i_{n}),(j_1,\dots ,j_{n})\in \prod_{k=1}^{n} \{1,\dots,a_k\}$
we have

\begin{enumerate}[(i)]
\item \label{002} $\conv(K_{i_1 \dots i_n})\cap \conv(K_{j_1\dots j_n})=\emptyset$,
\item \label{001} $K_{i_{1}\dots i_{n+1}}\subset K_{i_1 \dots i_{n}}$,
\item \label{003} $\lambda(K_{i_1\dots i_n})=\frac{\lambda(K)}{a_1\cdots a_n}$.
\end{enumerate}
We say that the $K_{i_1\dots i_n}$ are the \emph{elementary pieces of $K$}.
\end{definition}

\begin{definition} \label{d:Cantor2}
Let $\{ a_n \}_{n\in \NN^+}$, $\{b_n \}_{n\in \NN^+}$ be sequences of positive integers such that $a_n\geq b_n$ for all $n\in \NN^+$.
A compact set $C\subset \RR$ is an \emph{$(a_n,b_n)$-type Cantor set} if it is of the form
$$C=\bigcap_{n=1}^{\infty}\left(\bigcup_{i_1=1}^{b_1}\dots \bigcup_{i_n=1}^{b_n} C_{i_1 \dots i_n}\right),$$
where $C_{i_1\dots i_n}\subset \RR$ are compact sets and there is an $(a_n)$-type fat Cantor set $K\subset \RR$ of the form
\eqref{eq:defK} such that for all $n\in \NN^+$ and $(i_1,\dots,i_n)\in \prod_{k=1}^{n} \{1,\dots,b_k\}$
\begin{equation} \label{004} \emptyset \neq C_{i_{1}\dots i_{n+1}}\subset C_{i_1 \dots i_{n}}\subset K_{i_{1}\dots i_{n}}.
\end{equation}
The compact set $C\subset \RR$ is an \emph{$(a_n,b_n)$-type compact set} if it satisfies the above definition after replacing \eqref{002} by the weaker
property
\begin{enumerate}[(1)] \item \label{1star} $\conv(K_{i_1 \dots i_n})\cap \conv(K_{j_1\dots j_n})$ is either empty or a singleton.
\end{enumerate}
\end{definition}

\begin{example} The triadic Cantor set $C$ is an $(a_n,b_n)$-type compact set, where $a_n=3$ and $b_n=2$ for all $n$.
Indeed, for each $n$ let $\{K_{i_1\dots i_n}: 1\leq i_1,\dots, i_n\leq 3\}$ be the set of triadic intervals of $[0,1]$ of length
$3^{-n}$ and let $\{C_{i_1\dots i_n}: 1\leq i_1,\dots, i_n\leq 2\}$ be the set of the triadic intervals of
length $3^{-n}$ whose interior intersects $C$. Indexing $K_{i_1\dots i_n}$ and $C_{i_1\dots i_n}$ appropriately witnesses our claim.
\end{example}

For the following well-known lemma see e.g. \cite[Theorem~4.19]{MP}.

\begin{lemma}[Mass distribution principle] \label{l:fr} Let $\mu$ be a mass distribution on a metric space $X$. Assume
that there are $c,s,\delta\in \RR^+$ such that $\mu(B)\leq c (\diam B)^s$ for every Borel set
$B\subset X$ with $\diam B \leq \delta$. Then $\dim_H X\geq s$.
\end{lemma}

\begin{lemma} \label{l:Can} Let $C\subset \RR$ be an $(a_n,b_n)$-type compact set such that for all $n\in \NN^+$ we have
\begin{equation} \label{eq:ab} a_n\geq \left(\frac{a_1\cdots a_{n+1}}{b_1\cdots b_{n+1}}\right)^{n+1}.
\end{equation}
Then $\dim_H C=1$.
\end{lemma}

\begin{proof} Let $C_{i_1 \dots i_n}$ be the compact sets corresponding to Definition~\ref{d:Cantor2}.
Let $K\subset \RR$ be a compact set with elementary pieces $K_{i_1\dots i_n}$
associated to $C$. By considering similar copies of $C$ and $K$ we may assume that $\lambda(K)=1$.

For all $n\in \NN^+$ let $\iI_n=\prod_{k=1}^{n} \{1,\dots,b_k\}$.
We may suppose that $C_{i_1\dots i_n}\subset C$ for all $n\in \NN^+$ and
$(i_1,\dots,i_n)\in \iI_n$, otherwise we intersect them with $C$. Now we construct
a Borel probability measure $\mu$ supported on $C$ such that for all $n\in \NN^+$ and
$(i_1,\dots,i_n)\in \iI_n$ we have
\begin{equation} \label{defmu} \mu(C_{i_1 \dots i_n})=\frac{1}{b_1\cdots b_n}. \end{equation}
Choose $x_{i_1\dots i_n}\in C_{i_1\dots i_n}$ for all $n\in \NN^+$ and
$(i_1,\dots,i_n)\in \iI_n$. Define the probability measures
$$\mu_n=\sum_{(i_1,\dots,i_n)\in \iI_n} (b_1\dots b_n)^{-1}\delta_{x_{i_1\dots i_n}},$$
where $\delta_{x}$ denotes the Dirac measure concentrated on $\{x\}$. Let $F_n$ be the distribution function of $\mu_n$.
The definitions of $C$ and $\mu_n$ easily yield that $F_n$ converges (uniformly) to a continuous distribution function $F$.
Let $\mu$ be the Borel probability measure associated with $F$.
Then $\mu_n$ converges weakly to $\mu$ by \cite[Theorem~12.7]{MP}, so \cite[Theorem~12.6]{MP} yields that for all $n\in \NN^+$ and
$(i_1,\dots,i_n)\in \iI_n$ we have
$$\mu(C_{i_1\dots i_n})\geq \limsup_{k\to \infty} \mu_k(C_{i_1\dots i_n})=\frac{1}{b_1\cdots b_n}.$$
As $\mu$ is continuous, we have $\sum_{(i_1,\dots,i_n)\in \iI_n} \mu(C_{i_1\dots i_n})\leq 1$.
These imply that \eqref{defmu} holds and $\mu$ is supported on $C$.

Fix an arbitrary $k\in \NN^+$ and a Borel set $B\subset C$ with $\diam B\leq (a_1\cdots a_k)^{-1}$.
We can choose $n>k$ and $t\in \{1,\dots,a_n-1\}$ such that
\begin{equation} \label{eq:t+1}
\frac{t}{a_1\cdots a_n}\leq \diam B\leq \frac{t+1}{a_1\cdots a_{n}}.
\end{equation}
Property \eqref{003} yields that for all $n\in \NN^+$ and $(i_1,\dots,i_n)\in \prod_{k=1}^{n}\{1,\dots ,a_k\}$
$$\diam (\conv(K_{i_1\dots i_n}))\geq \lambda(K_{i_1\dots i_n})=\frac{1}{a_1\cdots a_n},$$
thus property \eqref{1star} and \eqref{eq:t+1} yield that $B$ can intersect at most $t+3$ sets of the form $K_{i_1\dots i_n}$.
Since $C_{i_1\dots i_n}\subset K_{i_1\dots i_n}$ for all $n\in \NN^+$ and $(i_1,\dots,i_n)\in \iI_n$,
we obtain that $B$ can intersect at most $t+3$ sets of the form
$C_{i_1\dots i_n}$. Therefore
\begin{equation} \label{eq:t+2} \mu(B)\leq \frac{t+3}{b_1\dots b_n}.\end{equation}
Inequalities \eqref{eq:ab} and $t+1\leq a_n$ yield
\begin{equation} \label{eq:an-1} \left(\frac{a_1\cdots a_n}{b_1\cdots b_n}\right)^n\leq a_{n-1}\leq \frac{a_1\cdots a_n}{t+1}. \end{equation}
Inequalities \eqref{eq:t+2}, \eqref{eq:an-1}, \eqref{eq:t+1} and $n>k$ with $\diam B\leq 1$ imply
\begin{align*}
\mu(B)&\leq \frac{t+3}{b_1\dots b_n} \\
&\leq \frac{4t}{a_1\cdots a_n}\cdot \frac{a_1\cdots a_n}{b_1 \cdots b_n} \\
&\leq 4(\diam B)  \left(\frac{a_1\cdots a_n}{t+1}\right)^{1/n} \\
&\leq 4(\diam B) (\diam B)^{-1/n} \\
&\leq 4(\diam B)^{1-1/k}.
\end{align*}
Thus Lemma~\ref{l:fr} yields that $\dim_H K\geq 1-1/k$.
As $k\in \NN^+$ was arbitrary, we obtain that $\dim_H K=1$. The proof is complete.
\end{proof}

\begin{lemma} \label{l:aa} Let $C\subset \RR$ be compact with $\lambda(C)>0$,
and let $\{a_n \}_{n\in \NN^+}$ be an arbitrary sequence of positive integers.
Then for every $\varepsilon>0$ there is an $(a_n)$-type fat Cantor set $K\subset C$ such that $\lambda(K)\geq \lambda(C)-\varepsilon$.
\end{lemma}

\begin{proof}
Let $\varepsilon>0$. It is straightforward to construct an $(a_n)$-type fat Cantor set
$D\subset [0,1]$ with elementary pieces $D_{i_1\dots i_n}\subset D$ such that $\lambda(D)\geq 1-\varepsilon$. By considering a similar copy of $C$
we may assume that $\lambda(C)=1$. Let $\phi \colon C\to [0,1]$ be the onto map defined as
$$\phi(x)=\lambda((-\infty,x]\cap C).$$
For every Borel set $B\subset [0,1]$ we have
\begin{equation} \label{eB} \lambda(\phi^{-1}(B))=\lambda(B),
\end{equation}
since \eqref{eB} holds for all intervals in $[0,1]$ by the definition of $\phi$, thus Carath\'eodory's extension theorem
yields that the Borel probability measures $\lambda \circ \phi^{-1}$ and $\lambda|_{[0,1]}$ coincide.

Let us define $K=\phi^{-1}(D)\subset C$ and $K_{i_1\dots i_n}=\phi^{-1}(D_{i_1\dots i_n})$ for all $n\in \NN^+$ and
$(i_1,\dots,i_n)\in \prod_{k=1}^{n}\{1,\dots ,a_k\}$. Applying
that $\phi$ preserves the order $\leq$ and \eqref{eB} yields that
$K$ is an $(a_n)$-type fat Cantor set with elementary pieces
$K_{i_1\dots i_n}$ such that $\lambda(K)=\lambda(D)\geq 1-\varepsilon=\lambda(C)-\varepsilon$.
\end{proof}

\begin{corollary} \label{c:Cantor} Let $K\subset \RR$ be a compact set with $\lambda(K)>0$ and let
$\{a_n \}_{n\in \NN^+}$ be an arbitrary sequence of positive integers. Then there exist $(a_n)$-type fat Cantor sets
$K_i\subset K$ such that $\lambda(\bigcup_{i=1}^{\infty} K_i)=\lambda(K)$.
\end{corollary}

\begin{lemma} \label{l:D} Let $G,H$ be abelian Polish groups and let $\Phi \colon G\to H$ be a continuous onto homomorphism.
If $S\subset H$ is prevalent then so is $\Phi^{-1}(S)\subset G$.
\end{lemma}

For the proof of the above lemma see \cite[Proposition~8.]{D}. The following corollary follows from Lemma~\ref{l:D} and the fact that
Tietze's extension theorem holds in $\RR^d$.

\begin{corollary} \label{c:hereditary} Let $K_1\subset K_2$ be compact metric spaces and let $d\in \NN^+$.
Define
$$R\colon C(K_2,\RR^d)\to C(K_1,\RR^d), \quad R(f)=f|_{K_1}.$$
If $\iA \subset C(K_1,\RR^d)$ is prevalent then so is $R^{-1}(\iA)\subset C(K_2,\RR^d)$.
\end{corollary}

\begin{lemma} \label{l:occup}
Let $K$ be a compact metric space and let $\mu$ be a mass distribution on $K$.
Let $K_n\subset K$ be compact sets with
$\mu(K)=\mu(\bigcup_{n=1}^{\infty} K_n)$. If $\Delta$ is an upward closed family of subsets of $K$ and
for all $n\in \NN^+$ the
\begin{align*}\iA_n=\{&f\in C(K_n,\RR^d): \exists \textrm{ an open set } U_f\subset \RR^d \textrm{ such that}\\
&\mu(f^{-1}(U_f))=\mu(K_n) \textrm{ and } f^{-1}(y)\in \Delta \textrm{ for all } y\in U_f\}
\end{align*}
are prevalent then so is
\begin{align*} \iA=\{&f\in C(K,\RR^d): \exists \textrm{ an open set } U_f\subset \RR^d \textrm{ such that}\\
&\mu(f^{-1}(U_f))=\mu(K) \textrm{ and } f^{-1}(y)\in \Delta \textrm{ for all } y\in U_f\}.
\end{align*}
\end{lemma}

We will apply the above lemma for families of the form $\Delta=\{A: \dim_H A\geq c\}$ and
$\Delta=\{A: \dim_H A\geq c_1 \textrm{ and } \dim_P A\geq c_2\}$.

\begin{proof}
For all $n\in \NN^+$ let
$$R_n\colon C(K,\RR^d)\to C(K_n,\RR^d),\quad R_{n}(f)=f|_{K_n}.$$
Corollary~\ref{c:hereditary} implies that the $R^{-1}_n(\iA_n)$ are prevalent in $C(K,\RR^d)$.
As a countable intersection of prevalent sets, $\bigcap_{n=1}^{\infty} R^{-1}_n(\iA_n)$ is also prevalent in $C(K,\RR^d)$. Thus it is enough to prove that
$\bigcap_{n=1}^{\infty} R^{-1}_n(\iA_n)\subset \iA$. For all
$f\in \bigcap_{n=1}^{\infty} R^{-1}_n(\iA_n)$ let $U_f=\bigcup_{n=1}^{\infty} U_{f|_{K_n}}$.
Then $U_f\subset \RR^d$ is open and for all $y\in U_f$ there is an $n\in \NN^+$ such that $y\in U_{f|_{K_n}}$. Then $f^{-1}(y) \supset (f|_{K_n})^{-1}(y)\in \Delta$ implies that $f^{-1}(y)\in \Delta$.

Finally, we need to show that $\mu(f^{-1}(U_f))=\mu(K)$. By $\mu(K)=\mu(\bigcup_{n=1}^{\infty}K_n)$ it is enough to prove
that $\mu( f^{-1}(U_f)\cap K_n)=\mu(K_n)$ for an arbitrary fixed $n\in \NN^+$. The definitions of $U_f$ and $\iA_n$ yield that
$$\mu\left(f^{-1}(U_f)\cap K_n\right)\geq \mu\left((f|_{K_n})^{-1}(U_{f|_{K_n}})\right)=\mu(K_n),$$
and the proof is complete.
\end{proof}

\begin{lemma} \label{l:Che} Let $u,v\in \NN^{+}$ and $0<p\leq 1/v$. Assume that $\xi_1,\dots,\xi_u$ are independent random variables such that $\Prob(\xi_i=j)=p$
for all $i\in \{1,\dots,u\}$ and $j\in \{1,\dots,v\}$. Then
$$\Prob\left(\# \{i: \xi_i=j\}< up/2 \textrm{ for some } j\in \{1,\dots,v\} \right)\leq \frac{4v}{up}.$$
\end{lemma}

\begin{proof} Let us fix $j\in \{1,\dots,v\}$ arbitrarily and for all $i\in \{1,\dots,u\}$ let $X_i=1$ if $\xi_i=j$, and let $X_i=0$ otherwise.
Set $X=\sum_{i=1}^{u} X_i$. Then $\mathbb{E}(X_i)=p$ and $\Var(X_i)=p-p^2<p$ for all $i\in \{1,\dots,u\}$, thus $\mathbb{E}(X)=up$ and the independence of $X_i$ yields
$\Var(X)=\sum_{i=1}^{u}\Var(X_i)<up$. Then Chebyshev's inequality \cite[(5.32)]{Bi} implies
\begin{align*} \Prob(\#\{i: \xi_i=j\}<up/2)&=\Prob(X<up/2) \\
&\leq \Prob\left(|X-\mathbb{E}(X)|>\mathbb{E}(X)/2\right) \\
&\leq \Var(X)/( \mathbb{E} (X)/2)^2 \leq \frac{4}{up}.
\end{align*}
Hence $$\Prob \left(\#\{i: \xi_i=j\}<up/2 \textrm{ for some } j\in \{1,\dots,v\} \right)\leq \frac{4v}{up},$$
and this concludes the proof.
\end{proof}

\begin{lemma} \label{l:pxy} If $X,Y$ are independent $\RR^d$-valued
random variables and $r>0$ then
$$\Pr(|X-Y|\leq r)\leq \sup_{y\in \RR^d} \Pr(|X-y|\leq r).$$
\end{lemma}
\begin{proof}
Let $\mu_X$, $\mu_Y$ and $\mu_{X,Y}$ be the distribution measure of $X$, $Y$ and $(X,Y)$, respectively.
The independence of $X$ and $Y$ yields $\mu_{X,Y}=\mu_X \times \mu_Y$, thus
\begin{align*}\label{eq:Pr} \Pr(|X-Y|\leq r)&=\iint_{\RR^{2d}} \chi_{\{(x,y):\, |x-y|\leq r\}} \,\mathrm{d} \mu_{X,Y}(x,y)  \\
&=\int_{\RR^d} \int_{\RR^d} \chi_{\{(x,y):\, |x-y|\leq r\}} \,\mathrm{d} \mu_X(x)\,\mathrm{d} \mu_Y(y) \\
&=\int_{\RR^d} \Pr(|X-y|\leq r) \,\mathrm{d} \mu_Y(y) \\
&\leq \sup_{y\in \RR^d} \Pr(|X-y|\leq r).
\end{align*}
The proof is complete.
\end{proof}

\section{Dimensions of fibers of prevalent continuous maps}

\subsection{The real case}\label{ss:real} First we prove the Main Theorem for $K\subset \RR$ and $\mu=\lambda$.

\begin{theorem} \label{t:real} Let $K\subset \RR$ be a compact set with $\lambda(K)>0$ and let $d\in \NN^+$.
Then for the prevalent $f\in C(K,\RR^d)$ there exists an open set $U_{f}\subset \RR^d$ such that $\lambda(f^{-1}(U_f))=\lambda(K)$ and
for all $y\in U_{f}$ we have
$$\dim_{H} f^{-1} (y)=1.$$
\end{theorem}

In the special case $K=[0,1]$ we obtain the following:

\begin{corollary} \label{c:occup} For the prevalent $f\in C[0,1]$ there is an open set $U_f\subset \RR$ such that $\lambda(f^{-1}(U_f))=1$ (hence $U_f$
is dense in $f([0,1])$) and for all $y\in U_f$ we have
$$\dim_H f^{-1}(y)=1.$$
\end{corollary}

\begin{proof}[Proof of Theorem~\ref{t:real}]
Consider
\begin{align*} \iA=\{&f\in C(K,\RR^d): \exists \textrm{ an open set } U_f\subset \RR^d \textrm{ such that}\\
&\lambda(f^{-1}(U_f))=\lambda (K) \textrm{ and } \dim_H f^{-1}(y)=1 \textrm{ for all } y\in U_f\}.
\end{align*}
Lemma~\ref{l:co} with $c=1$ yields that $\iA$ is co-analytic. By Theorem~\ref{t:S} it is enough to show that there exists a Borel probability measure
$\mu$ on $C(K,\RR^d)$ such that $\mu(\iA-g)=1$ for all $g\in C(K,\RR^d)$.

Now we construct the measure $\mu$. Let us endow $\RR^d$ with the maximum norm, which we simply denote by $|\cdot|$.
Let $s=2^d$ and let $S_n=\{-2^{-n},2^{-n}\}^{d}$ for all $n\in \NN^+$, then $\#S_n=s$. Clearly for all $z\in \RR^d$ and $n\in \NN$
we obtain that
\begin{equation} \label{eq:cov} B(z, 2^{-n})=\bigcup_{y\in  S_{n+1}} B(z + y, 2^{-(n+1)}).
\end{equation}
For all $n\in \NN^+$ let us define the positive integers $a_n$ and $b_n$ by
$$a_n=(2s)^{4^{n}} \quad \textrm{and} \quad b_n=(2s)^{-(n+3)}a_n,$$
easy calculations show that there is an $n_0\in \NN^+$ such that for all $n\geq n_0$ we have
\begin{equation} \label{eq:a_n} a_n\geq \max\left\{(2s)^{8n}(a_1\cdots a_{n-1}),
\left(\frac{a_1\cdots a_{n+1}}{b_1\cdots b_{n+1}}\right)^{n+1} \right\}.
\end{equation}
For all $n\in \NN^{+}$ let
$$\iI_n=\prod_{i=1}^{n} \{1,\dots,a_i\}.$$
Let us recall Definition~\ref{d:Cantor1}.
Corollary~\ref{c:Cantor} implies that there exist $(a_n)$-type fat Cantor sets $K_i\subset K$ such that
$\lambda(\bigcup_{i=1}^{\infty}K_i)=\lambda(K)$. Therefore, by Lemma~\ref{l:occup} we may assume that $K$ is an $(a_n)$-type fat Cantor set with elementary
pieces $K_{i_1\dots i_n}$. By considering a similar copy of $K$ we may suppose that $\lambda(K)=1$.
Then for all $n\in \NN^+$ and $(i_1,\dots,i_n)\in \iI_n$ we have
$$\lambda (K_{i_1\dots i_n})=\frac{1}{a_1\cdots a_n}.$$
For all Borel sets $A,B\subset K$ with $\lambda(B)>0$ let us use the notation
$$\lambda(A\,|\, B)=\frac{\lambda(A\cap B)}{\lambda(B)}.$$
For all $n\in \NN^{+}$ and $(i_1,\dots,i_n)\in \iI_n$ let us define countably many independent random variables
$X_{i_1\dots i_n}$ and $Y_{i_1\dots i_n}$ such that for all $y\in S_n$ we have
\begin{equation} \label{eq:XY} \Pr(X_{i_1\dots i_n}=y)=\Pr(Y_{i_1\dots i_n}=y)=1/s.
\end{equation}
For every $n\in \NN^+$ and $x\in K$ there exists a unique $(i_1,\dots,i_n)\in \iI_n$ for which $x\in K_{i_1\dots i_n}$. Then
let us define the random function $f_n\in C(K,\RR^d)$ as
$$f_n(x)=X_{i_1\dots i_n}-Y_{i_1\dots i_n}.$$
Note that the dependence of the right hand side on $x$ is simply that the indices depend on $x$.
Let $\mathbb{P}_n$ be the probability measure on $C(K,\RR^d)$
corresponding to this method of randomly choosing $f_n$, and let $\iR_n\subset C(K,\RR^d)$ be its finite support.
Clearly for all $f_n\in \iR_n$ and $x\in K$ we have
\begin{equation} \label{fnx} |f_n(x)|\leq 2^{1-n}.
\end{equation}
Thus $\sum_{n=1}^{\infty} f_{n}$ always converges uniformly. Let
$\mathbb{P}=\prod_{n=1}^{\infty} \mathbb{P}_n$ be a probability measure on the Borel subsets of $\iR=\prod_{n=1}^{\infty}\iR_n$
and let
$$\pi\colon \iR \to C(K,\RR^d), \quad  \pi((f_n))=\sum_{n=1}^{\infty} f_n.$$
Let us define
$$\mu=\mathbb{P}\circ \pi^{-1}.$$

Now we prove that $\mu(\iA-g)=1$ for all $g\in C(K,\RR^d)$. More precisely, we will show that for each $g\in C(K,\RR^d)$
the stochastic process $W=\sum_{n=1}^{\infty} f_n$ satisfies $g+W\in \iA$ almost surely. Let $g\in C(K,\RR^d)$ and $\varepsilon>0$ be arbitrarily fixed,
it is enough to show that $\mu(\iA-g)\geq 1-\varepsilon$. As $g(K)$ is compact, we can fix an integer $m>n_0$ such that
$2^m>1/\varepsilon$ and $g(K)$ can be covered by $2^m$ closed balls of radius $1$, it is sufficient to prove that
\begin{equation} \label{eq:iA-g} \mu(\iA-g)=\mathbb{P}(\pi^{-1}(\iA-g))\geq 1-2^{-m}.
\end{equation}
For all $n\in \NN$ consider
\begin{align*}  h_n&=g+\sum_{i=1}^{m+n} f_i, \\
r_n&=(2s)^{-(m+n+2)}, \\
p_{n}&=\frac{(2s)^{m+n+5}}{a_{m+n+1}}.
\end{align*}
\begin{statement}  \label{st:1} Let $n\in \NN$ and
assume that $z\in \RR^d$ and for all $i\in \{1,\dots,m+n\}$ the functions
$f_{i}\in \iR_{i}$ and $\sigma \in \iI_{m+n}$ are fixed. Let $A\subset h_n^{-1}(B(z,2^{-(m+n)}))$
with
$$\lambda(A\, |\,  K_{\sigma})\geq r_n.$$
For all $y \in S_{m+n+1}$ let us define the random set $I(y)\subset \{1,\dots,a_{m+n+1}\}$ as
$$I(y)=\left\{i: \lambda\left(A\cap h_{n+1}^{-1}\left(B\left(z+y,2^{-(m+n+1)}\right)\right)\, \big| \,
K_{\sigma i}\right)\geq r_{n+1}\right\},$$
where $\sigma i$ is the concatenation of $\sigma$ and $i$. Then
$$\mathbb{P}_{m+n+1}\left(\#I(y)<b_{m+n+1} \textrm{ for some } y\in S_{m+n+1}\right)\leq p_{n}.$$
\end{statement}
\begin{proof}[Proof of Statement \ref{st:1}]
Let us define $I\subset \{1,\dots,a_{m+n+1}\}$ as
$$I=\{i: \lambda(A\, |\,  K_{\sigma i})\geq r_n/2\}.$$
First we prove that
\begin{equation} \label{eq:numI} \#I\geq \frac{r_n a_{m+n+1}}{2}. \end{equation}
Our assumption and the definition of $I$ imply that
\begin{align*}
r_n \lambda(K_{\sigma}) &\leq \lambda (A\cap K_{\sigma})=\sum_{i=1}^{a_{m+n+1}}\lambda(A\cap K_{\sigma i}) \\
&\leq \sum_{i=1}^{a_{m+n+1}} \frac{r_n}{2} \lambda(K_{\sigma i})+\sum_{i\in I} \lambda(K_{\sigma i}) \\
&= \frac{r_n}{2} \lambda(K_{\sigma})+(\#I) \frac{\lambda(K_{\sigma})}{a_{m+n+1}}, \\
\end{align*}
which easily yields \eqref{eq:numI}. Then $A\subset h_n^{-1}(B(z,2^{-(m+n)}))$ and \eqref{eq:cov} imply that
$$A\subset \bigcup_{y\in S_{m+n+1}} h_n^{-1}\left(B\left(z+y,2^{-(m+n+1)}\right)\right).$$
Thus the definition of $I$ and $r_{n+1}=r_{n}/(2s)$ yield that for all $i\in I$ there exists
$y(i)\in S_{m+n+1}$ such that
\begin{equation} \label{eq:nuh}
\lambda\left(A\cap h_n^{-1}\left(B\left(z+y(i),2^{-(m+n+1)}\right)\right)\, \big| \,  K_{\sigma i}\right)\geq r_{n+1}.
\end{equation}
Let $S_{m+n+1}=\{y_j: 1\leq j\leq s\}$. Define for all $i\in I$ independent random variables
\begin{equation} \label{eq:xi} \xi_i= \begin{cases}
j \textrm{ if } X_{\sigma i}=y_j \textrm{ and }  Y_{\sigma i}=y(i), \\
0 \textrm{ otherwise}.
\end{cases}
\end{equation}
For all $j\in \{1,\dots, s\}$ let us define the random set
$$I_j=\{i\in I: \xi_i=j\}.$$
Now we show that for all $j\in \{1,\dots, s\}$ we have
\begin{equation} \label{eq:2Ij} I_j\subset I(y_j).
\end{equation}
Assume that $i\in I_j$ and $x\in K_{\sigma i}$, then
$$h_{n+1}(x)=h_{n}(x)+f_{n+1}(x)=h_n(x)+ X_{\sigma i}-Y_{\sigma i}=h_n(x)+y_j-y(i),$$
therefore
\begin{equation} \label{eq:hh}   h_n^{-1}\left(B\left(z+y(i),2^{-(m+n+1)}\right)\right)\cap K_{\sigma i} \subset h_{n+1}^{-1}\left(B\left(z+y_j,2^{-(m+n+1)}\right)\right).
\end{equation}
Formulae \eqref{eq:hh} and \eqref{eq:nuh} imply
\begin{equation} \label{eq:zzz}
\lambda \left(A\cap h_{n+1}^{-1}\left(B\left(z+y_j,2^{-(m+n+1)}\right)\right)\, \big| \,  K_{\sigma i}\right)\geq r_{n+1},
\end{equation}
thus $i\in I(y_j)$, so \eqref{eq:2Ij} holds. The definitions yield that

\begin{equation} \label{eq:tedious}
\frac{r_n a_{m+n+1}}{4s^2}=b_{m+n+1} \quad \textrm{and} \quad \frac{8s^3}{r_n a_{m+n+1}}=p_{n}.
\end{equation}

Clearly, we have $\Prob (\xi_i=j)=1/s^2$ for all $i\in I$ and $j\in \{1,\dots, s\}$.
We apply Lemma~\ref{l:Che} for $\xi_i$ with $u=\#I$, $v=s$ and $p=1/s^2$. Then \eqref{eq:numI} and the first part of
\eqref{eq:tedious} yield that $b_{m+n+1}\leq up/2$. Therefore \eqref{eq:2Ij}, Lemma~\ref{l:Che}, \eqref{eq:numI}
and the second part of \eqref{eq:tedious} imply that
\begin{align*}
\mathbb{P}_{m+n+1}&\left(\#I(y)< b_{m+n+1} \textrm{ for some } y\in S_{m+n+1} \right) \\
&\leq \Prob\left(\#I_j< b_{m+n+1} \textrm{ for some } j\in \{1,\dots, s\} \right) \\
&\leq \Prob \left(\#\{i\in I: \xi_i=j\}<up/2 \textrm{ for some } j\in \{1,\dots,v\}  \right) \\
&\leq \frac{4v}{up}\leq \frac{8s^3}{r_n a_{m+n+1}}=p_{n}.
\end{align*}
The proof of the statement is complete.
\end{proof}

Now we return to the proof of Theorem~\ref{t:real}. For all $k\in \NN$ let $\iY_{k,k}=\iJ_{k,k}=\{ \emptyset\}$. For all natural numbers
$k<n$ let
$$\iY_{k,n}=\prod_{i=k+1}^{n} S_{m+i} \quad \textrm{and} \quad \iJ_{k,n}=\prod_{i=k+1}^{n} \{1,\dots,b_{m+i}\}.
$$
Let $(f_1,\dots,f_m)\in  \iR_1\times \dots \times \iR_m$ be arbitrary. By the definition of $m$ we can cover $B(g(K),2)$
by $2^m$ closed balls of radius $3$, and so by
$2^ms^{m+2}<(2s)^{m+2}$ closed balls of radius $2^{-m}$. Let $T_m\subset \RR^d$ be the set of their centers, that is,
$\#T_m\leq (2s)^{m+2}$ and $B(g(K),2) \subset \bigcup_{y\in T_m} B(y,2^{-m})$. Let
$$\iT_0=\{y_0\in T_m: \lambda(h_0^{-1}(B(y_0,2^{-m})))\geq r_0\}.$$
As \eqref{fnx} yields $h_0=g+\sum_{i=1}^{m} f_i\in B(g,2)$, we have $h_0(K)\subset \bigcup_{y\in T_m} B(y,2^{-m})$. Therefore
$r_0=(2s)^{-(m+2)}$ and $\#T_m\leq (2s)^{m+2}$ imply that $\iT_0\neq \emptyset$.
Let us fix $(f_1,\dots,f_m)$ and for all $y_0\in \iT_0$ fix $\sigma_{y_0}\in \iI_m$ such that
$$\lambda(h_0^{-1}(B(y_0,2^{-m}))\,| \, K_{\sigma_{y_0}})\geq r_0.$$

Let $n\in \NN$ and suppose that $f_i\in \iR_i$ are fixed for all $i\in \{1,\dots,m+n\}$. Assume by induction that
for all $k\in \{0,\dots,n\}$ the sets $\iT_k\subset T_m\times \prod_{i=1}^{k} S_{m+i}$ and for all
$(y_0,\dots,y_k)\in \iT_k$, $(y_{k+1},\dots,y_n)\in \iY_{k,n}$ and $(j_{k+1},\dots,j_n)\in \iJ_{k,n}$ the index sets
\begin{equation} \label{eq:sigma} \sigma=\sigma_{y_0\dots y_k}(y_{k+1},\dots,y_n;j_{k+1},\dots,j_n)\in \iI_{m+n}
\end{equation}
are already defined (we use the convention that $\sigma_{y_0\dots y_n}(\emptyset;\emptyset)=\sigma_{y_0\dots y_n}$)
such that
\begin{align} \label{eq:hn+1}
\begin{split}
&\lambda(A_k^{n}(y_0,\dots ,y_n) \, |\,   K_{\sigma})\geq r_{n}, \textrm{ where}  \\
&A_k^{n}(y_0,\dots ,y_n)=\bigcap_{i=k}^{n} h_{i}^{-1}\left(B\left((y_0+\dots+y_i),2^{-(m+i)}\right)\right).
\end{split}
\end{align}
Let us consider the functions $f_{m+n+1}\in \iR_{m+n+1}$ for which for every $\sigma$ from \eqref{eq:sigma} and
for every $y_{n+1}\in S_{m+n+1}$ and $j_{n+1}\in \{1,\dots,b_{m+n+1}\}$ we can define
$i$ as the $j_{n+1}$st smallest element of $\{1,\dots,a_{m+n+1}\}$ that satisfies
\begin{equation*}
\lambda\left(A_k^{n}(y_0,\dots, y_n)\cap h_{n+1}^{-1}\left(B\left((y_0+\dots+ y_{n+1}),2^{-(m+n+1)}\right)\right)\, \big| \,  K_{\sigma i} \right)\geq r_{n+1},
\end{equation*}
and let us define
\begin{equation*} \label{eq:deffi} \sigma_{y_0\dots y_k}(y_{k+1},\dots,y_{n+1};j_{k+1},\dots,j_{n+1})=\sigma i\in \iI_{m+n+1}.
\end{equation*}
Statement~\ref{st:1} implies that the $\mathbb{P}_{m+n+1}$-probability that $f_{m+n+1}\in \iR_{m+n+1}$ does not have this property is at most
\begin{align} \label{eq:sumk}
\begin{split}
\sum_{k=0}^{n}(\#\iT_k)(\#\iY_{k,n})(\#\iJ_{k,n})p_{n}&\leq \sum_{k=0}^{n} (2s)^{m+2}s^{n} (b_{m+k+1}\cdots b_{m+n})p_n  \\
&\leq (2s)^{m+n+2}(b_{1}\cdots b_{m+n})p_n  \\
&\leq (2s)^{7(m+n+1)}(a_{1}\cdots a_{m+n})(a_{m+n+1})^{-1}\\
&\leq 2^{-(m+n+1)},
\end{split}
\end{align}
where we used $n+1\leq 2^n$, $b_n \le a_n$ and \eqref{eq:a_n}.
Let us fix $f_{m+n+1}$ with the above property and define $\iT_{n+1}\subset T_m\times \prod_{i=1}^{n+1} S_{m+i}$ as
$$\iT_{n+1}=\left\{(y_0,\dots,y_{n+1}): \lambda\left(h_{n+1}^{-1}\left(B\left((y_0+\dots+y_{n+1}),2^{-(m+n+1)}\right)\right)\right)\geq r_{n+1}\right\},$$
and for all $(y_0,\dots,y_{n+1})\in \iT_{n+1}$ let us fix $\sigma_{y_0\dots y_{n+1}}\in \iI_{m+n+1}$ such that
$$\lambda\left(h_{n+1}^{-1}\left(B\left((y_0+\dots+y_{n+1}),2^{-(m+n+1)}\right)\right)\, \big| \, K_{\sigma_{y_0\dots y_{n+1}}} \right)\geq r_{n+1}.$$
Let $\iF$ be the set of sequences $(f_i)\in \iR$ which can be defined by this process.
Then \eqref{eq:sumk} yields
\begin{equation} \label{eq:bbb} \mathbb{P}(\iF)\geq 1-\sum_{n=0}^{\infty} 2^{-(m+n+1)}=1-2^{-m}. \end{equation}
For all $(f_i)\in \iR$ let $h=g+\sum_{i=1}^{\infty} f_i$ and for all $(f_i)\in \iF$ let
$$U_h=\bigcup_{k=0}^{\infty} \bigcup_{(y_0,\dots,y_k)\in \iT_k} U\big((y_0+\dots+y_k),2^{-(m+k)}\big).$$

Now we are ready to prove \eqref{eq:iA-g}. By \eqref{eq:bbb} it is enough to show that for $\mathbb{P}$-almost every $(f_i)\in \iF$
we have $(f_i)\in \pi^{-1}(\iA-g)$. Therefore it is sufficient to prove that $\dim_H h^{-1}(y)=1$ for every $(f_i)\in \iF$ and $y\in U_h$,
and $\lambda(h^{-1}(U_h))=\lambda(K)$ for $\mathbb{P}$-almost every $(f_i)\in \iF$. Therefore
the following three lemmas will complete the proof of
Theorem~\ref{t:real}.

\begin{lemma} For all $(f_i)\in \iF$ and $y\in U_h$ we have $\dim_H h^{-1}(y)=1$.
\end{lemma}

\begin{proof}
Fix $k\in \NN$ and $(y_0,\dots,y_k)\in \iT_k$ such that
$y\in U((y_0+\dots+y_k),2^{-(m+k)})$. Then \eqref{eq:cov} yields that for all $n\in \NN^+$ we can fix $y_{k+n}\in S_{m+k+n}$ such that
\begin{equation} \label{eq:yin} y\in B\big((y_0+\dots +y_{k+n}),2^{-(m+k+n)}\big).
\end{equation}
For all $n\in \NN^+$ and $(j_1,\dots, j_n)\in \iJ_{k,k+n}=\prod_{i=1}^{n}\{1,\dots,b_{m+k+i}\}$ let
$$\sigma(j_1,\dots,j_n)=\sigma_{y_0\dots y_k}(y_{k+1},\dots,y_{k+n};j_1,\dots,j_n)\in \iI_{m+k+n}$$
and
$$C_{j_1 \dots j_n}=K_{\sigma(j_1,\dots,j_n)}\cap A_k^{k+n}(y_0,\dots, y_{k+n}),$$
where we recall \eqref{eq:sigma} and \eqref{eq:hn+1}. Let us define
\begin{equation*}C=\bigcap_{n=1}^{\infty} \left(\bigcup_{j_1=1}^{b_{m+k+1}}\dots \bigcup_{j_n=1}^{b_{m+k+n}} C_{j_1 \dots j_{n}}\right).
\end{equation*}
Then \eqref{eq:hn+1} yields that $h_{k+n}(C_{j_1 \dots j_n})\subset  B((y_0+\dots+y_{k+n}),2^{-(m+k+n)})$ for all $n\in \NN^+$ and
$(j_1,\dots,j_n)\in \iJ_{k,k+n}$. Thus \eqref{eq:yin} and uniform convergence imply $h(x)=y$ for all $x\in C$, so
\begin{equation} \label{eq:sups} C\subset h^{-1}(y).
\end{equation}

Let us recall Definitions~\ref{d:Cantor1} and \ref{d:Cantor2}.
Let $c_n=a_{m+k+n}$ and $d_n=b_{m+k+n}$ for all $n\in \NN^+$.
As $K$ is an $(a_n)$-type fat Cantor set, the $(m+k)$th level elementary pieces of $K$ are
$(c_n)$-type fat Cantor sets. Inequality~\eqref{eq:hn+1} yields that $C_{j_1 \dots j_n}\neq \emptyset$ for all $n\in \NN^+$ and
$(j_1,\dots,j_n)\in \prod_{i=1}^{n}\{1,\dots,d_i\}$,
thus $K$ witnesses that $C$ is a $(c_n,d_n)$-type Cantor set. Then \eqref{eq:a_n} easily implies that
$$c_n\geq \left(\frac{c_{1}\cdots
c_{n+1}}{d_{1}\cdots d_{n+1}}\right)^{n+1},$$
therefore Lemma~\ref{l:Can} yields that $\dim_H C=1$. Hence \eqref{eq:sups} implies that
\begin{equation} \label{eq:Hh} \dim_H h^{-1}(y)\geq \dim_H C=1,
\end{equation}
which completes the proof.
\end{proof}

\begin{lemma} Let $(f_i)\in \iF$ be fixed such that
$\lambda\circ h^{-1}$ is absolutely continuous with respect to $\lambda^d$. Then $\lambda(h^{-1}(U_h))=\lambda(K)$.
\end{lemma}

\begin{proof}  There exists a measurable function
$\varphi\colon \RR^d\to [0,\infty)$ such that for every Borel set $B\subset \RR^d$ we have
\begin{equation} \label{eq:rdens} \lambda(h^{-1}(B))=\int_{B} \varphi(y) \, \mathrm{d} \lambda^d(y).
\end{equation}
Assume to the contrary that $\lambda(h^{-1}(h(K)\setminus U_h))>0$. Then there exist a measurable set
$E\subset h(K)\setminus U_h$ and $c>0$ such that $\lambda^d(E)>0$ and $\varphi|_{E}\geq c$.
By Lebesgue's density theorem \cite[261D]{Fr} we can fix a
density point $z\in E$. As $z$ is a density point of $E$, there is an $n\in \NN^+$ such that $r=2^{-(m+n)}$ satisfies $r\leq c$ and
if $B\subset B(z,2r)$ is a ball with radius at least $r/4$ then
\begin{equation} \label{eq:ID} \lambda^d(B\cap E)\geq \lambda^d(B)/2.
\end{equation}
As $z\in E\subset h(K)\subset B(g(K),2)$, there exists
$(y_0,\dots,y_n)\in T_{m}\times S_{m+1}\times \dots \times S_{m+n}$ with $z\in B((y_0+\dots +y_n),r)$. Let
$y=y_0+\dots+y_n$ and $D=B(y,r/2)$, then we have $B(D,r/2)=B(y,r)\subset B(z,2r)$. Then \eqref{eq:rdens}, the definition of $c$ and \eqref{eq:ID} yield
\begin{equation} \label{eq:h1I}
\lambda(h^{-1}(D))=\int_{D} \varphi(y) \, \mathrm{d} \lambda^d(y) \geq c \lambda^d(D\cap E)\geq (c/2)\lambda^d(D).
\end{equation}
Let $\iB=\{B(y+y_{n+1}+y_{n+2},r/4): y_{n+i}\in S_{m+n+i}\}$, clearly $\#\iB=s^{2}$.
If $x\in K$ and $h(x)\in D$ then \eqref{fnx} and the definition of $D$ imply
$$h_{n+2}(x)\in B(h(x),2^{-(m+n+1)})\subset B(D,r/2)=B(y,r)=\bigcup \iB.$$
Hence there exists a $B\in \iB$ such that
$\lambda(h_{n+2}^{-1}(B))\geq s^{-2} \lambda(h^{-1}(D))$. Then \eqref{eq:h1I},
$\lambda^d(D)=s^{-(m+n)}$ and $c\geq 2^{-(m+n)}$ yield
$$\lambda(h_{n+2}^{-1}(B))\geq \frac{\lambda(h^{-1}(D))}{s^2}\geq  \frac{c\lambda^d(D)}{2s^2}>(2s)^{-(m+n+4)}=r_{n+2}.$$
Let $y_0+\dots+y_{n+2}$ be the center of $B$. The definition of $\iT_{n+2}$ and
$\lambda(h_{n+2}^{-1}(B))> r_{n+2}$ imply that $(y_0,\dots,y_{n+2})\in \iT_{n+2}$.
Hence the definition of $U_h$ yields that $B\subset U_h$. Then \eqref{eq:ID} implies that $\lambda^d(B\cap E)\geq \lambda^d(B)/2>0$,
thus $E\cap B\neq \emptyset$, so $E\cap U_h\neq \emptyset$.
This contradicts the definition of $E$, thus $\lambda(h^{-1}(U_h))=\lambda(K)$.
\end{proof}

\begin{lemma} For $\mathbb{P}$-almost every $(f_i)\in \iR$ the measure $\lambda\circ h^{-1}$ is absolutely continuous with respect to
$\lambda^d$.
\end{lemma}

\begin{proof}
For all $(f_i)\in \iR$ and $n\in \NN^+$ let $F_n=\sum_{i=n}^{\infty} f_i$.
For all distinct $x,z\in K$ let $i(x,z)$ be the minimal number $i$ such that
$x$ and $z$ are in different $i$th level elementary pieces of $K$. Fix $n\in \NN^+$ and $x,z\in K$ with $i(x,z)=n$ and also fix $r>0$.
Pick indexes $\{i_k\}_{k\geq 1}$ such that
$x\in K_{i_1\dots i_k}$ for all $k\in \NN^+$ and define
$$X_{n}(x)=\sum_{k=n}^{\infty} X_{i_1\dots i_k} \quad \textrm{and} \quad Y_{n}(x)=\sum_{k=n}^{\infty} Y_{i_1\dots i_k},$$
where we recall \eqref{eq:XY}. Then clearly $X_{n}(x)$ is uniformly distributed in
$B(\mathbf{0},2^{1-n})$, therefore for all $y\in \RR^d$ we have
\begin{equation} \label{eq:unif} \mathbb{P}(|X_n(x)-y|\leq r)\leq (r2^{n-1})^d\leq (r2^{n})^d.
\end{equation}
Then $i(x,z)=n$ implies that $X_{n}(x)$ and $Y_n(x)-g(x)+F_{n}(z)+g(z)$ are independent. Since $f_i(x)=f_i(z)$ for all $i<n$, the difference of
the above two variables equals $h(x)-h(z)$. Thus Lemma~\ref{l:pxy} and \eqref{eq:unif} imply
\begin{align} \label{eq:Phx}
\begin{split}
\mathbb{P}(|h(x)-h(z)|\leq r)&= \mathbb{P}(|X_{n}(x)-(Y_n(x)-g(x)+F_{n}(z)+g(z))|\leq r)  \\
&\leq\sup_{y\in \RR^d} \mathbb{P}(|X_n(x)-y|\leq r) \leq (r2^{n})^d.
\end{split}
\end{align}
Let $A_n=\{(x,z)\in K\times K: i(x,z)=n\}$ for all $n\in \NN^+$, then
\begin{equation} \label{eq:aprod} (\lambda\times \lambda)(A_n)=\frac{a_1\cdots a_n(a_n-1)}{(a_1\cdots a_n)^2}\leq \frac{1}{a_1\cdots a_{n-1}},
\end{equation}
where $a_1\cdots a_0=1$ by convention. Let us use the notation $\lambda_h=\lambda\circ h^{-1}$ and
define the random function $q\colon \RR^d\to [0,\infty]$ as
$$q(y)=\liminf_{r\to 0+} \frac{\lambda_h(B(y,r))}{\lambda^d(B(y,r))}.$$
By \cite[Theorem~2.12]{Ma} it is enough to show that, almost surely,
$q(y)<\infty$ for $\lambda_h$ almost every $y\in \RR^d$.
Therefore it is enough to prove that the following expected value is finite.
Applying Fatou's lemma, the substitution formula $\int_{\RR^d} \psi \, \mathrm{d} \lambda_h = \int_K \psi \circ h \, \mathrm{d} \lambda$, Fubini's theorem, \eqref{eq:Phx}, \eqref{eq:aprod} and
$a_n\geq s^{2n}=2^{2dn}$ yield
\begin{align*} \mathbb{E} \int_{\RR^d} q(y) \, \mathrm{d} \lambda_h(y)&\leq  \liminf_{r\to 0+} \frac{1}{(2r)^d}  \mathbb{E}  \int_{\RR^d} \lambda_h(B(y,r))\, \mathrm{d} \lambda_h(y) \\
&=  \liminf_{r\to 0+} \frac{1}{(2r)^d} \int_{K} \int_{K} \mathbb{P}(|h(x)-h(z)|\leq r) \,\mathrm{d} \lambda(x) \, \mathrm{d} \lambda(z) \\
&=  \liminf_{r\to 0+} \frac{1}{(2r)^d} \sum_{n=1}^{\infty} \iint_{A_n}
\mathbb{P}(|h(x)-h(z)|\leq r) \,\mathrm{d} \lambda(x) \, \mathrm{d} \lambda(z) \\
&\leq \liminf_{r\to 0+} \frac{1}{(2r)^d} \sum_{n=1}^{\infty} \frac{(r2^{n})^d}{a_1\cdots a_{n-1}}\\
&= \sum_{n=0}^{\infty} \frac{2^{dn}}{a_1\cdots a_{n}}\leq \sum_{n=0}^{\infty} 2^{-dn}<\infty,
\end{align*}
and the proof is complete.
\end{proof}

Therefore the proof of Theorem~\ref{t:real} is also complete.
\end{proof}

\subsection{The ultrametric case}\label{ss:ultra}

Before turning to the general case, we prove the Main Theorem for ultrametric spaces.
The key step is the following lemma, where the construction of the map $h$ is based on the proof of \cite[Theorem~2.1]{KMZ}.

\begin{lemma} \label{l:Holder}
Let $(K,d)$ be a compact ultrametric space. If $\mu$ is a continuous mass distribution on $K$, then there is a map
$h\colon K\to [0,1]$ and there are compact sets $K_n\subset K$ such that if $h_n=h|_{K_n}$ and $D_n=h(K_n)$ then for all $n\in \NN^+$
we have
\begin{enumerate}[(i)]
\item \label{e1} $\mu(\bigcup_{n=1}^{\infty} K_n)=\mu(K)$ and $\lambda(D_n)>0$;
\item \label{e2} the maps $h_n\colon K_n\to D_n$ are homeomorphisms;
\item \label{e3} for every Borel set $B\subset D_n$ we have $\mu(h_n^{-1}(B))=\lambda(B)$;
\item \label{e4} for every non-empty $B\subset D_n$ we have $\dim_H h_n^{-1}(B)\geq \dim_H \mu \cdot\dim_H B$;
\item \label{e5} for every non-empty $B\subset D_n$ we have $\dim_P h_n^{-1}(B)\geq \dim_P \mu \cdot\dim_H B$.
\end{enumerate}
\end{lemma}

\begin{remark} We may assume that the compact sets in the above lemma satisfy $K_n\subset K_{n+1}$ for all $n\geq 1$, but we will not use this property later.
\end{remark}

\begin{proof} We may assume that $\mu(K)=1$. By \cite[Lemma~2.3]{KMZ} we obtain that $K$ is a
$1$-monotone metric space, that is, there exists a linear order $\prec$ on $K$ such that $\diam [a,b]=d(a,b)$ for all $a,b\in K$, where $[a,b]$
denotes the closed interval $\{x\in K: a\preceq x\preceq b\}$.
It is easy to show (see also in \cite{NZ}) that each open interval $(a,b)=\{x\in K: a\prec x\prec b\}$ and every open half-line $(-\infty,x)=\{z\in K: z\prec x\}$, $(x, \infty)=\{z\in K: x\prec z\}$ is open in $K$, so every interval is Borel.
Let us define
$$h\colon K\to [0,1], \quad h(x)=\mu((-\infty,x)).$$
Then for every Borel set $B\subset [0,1]$ we have
\begin{equation} \label{eq:mert} \mu(h^{-1}(B))=\lambda(B),
\end{equation}
since it holds for intervals in $[0,1]$ by the definition of $h$, so Carath\'eodory's extension theorem yields that the Borel
probability measures
$\mu\circ h^{-1}$ and $\lambda|_{[0,1]}$ coincide.

Now we show that $h(K)=[0,1]$. The continuity of $\mu$ implies that $h$ is continuous, so $h(K)$ is compact.
Since $K$ is compact and the intervals of the form $(-\infty, b)$ and $(a,\infty)$ are open, there exists a minimal element $x_{-}$
and a maximal element $x_{+}$ in $(K,\prec)$. Then $h(x_{-})=\mu(\emptyset)=0$, and the continuity of $\mu$ yields
$h(x_{+})=\mu((-\infty,x_{+}))=\mu(K\setminus \{x_{+}\})=\mu(K)=1$.
Hence $\{0,1\}\subset h(K)$ and $h(K)$ is a subset of $[h(x_{-}),h(x_{+})]=[0,1]$.
Thus it is enough to prove that there are no $u,v\in h(K)$ with $u<v$ and $(u,v)\cap h(K)=\emptyset$.
Assume to the contrary that there exist such $u$ and $v$. Let $S$ be a countable dense subset of $K$ and let
$S_1=\{s\in S:h(s)\leq u\}$ and $S_2=\{t\in S:h(t)\geq v\}$. Then $(u,v)\cap h(K)=\emptyset$ implies that $S=S_1\cup S_2$.
Since $S_1$ and $S_2$ are countable and $h(x)=\mu((-\infty,x))$, we obtain $\mu(\bigcup_{s\in S_1} (-\infty,s))\leq u$ and $\mu(\bigcap_{t\in S_2} (-\infty,t))\geq v$. Let $E=\bigcap_{t\in S_2} (-\infty,t)\setminus \bigcup_{s\in S_1} (-\infty,s)$, then we have $\mu(E)\geq v-u>0$. But $E$ can contain at most two points: If $a,b,c\in E$ and $a\prec b\prec c$, then $(a,c)$ would be a non-empty open set not containing any point of the dense set $S=S_1\cup S_2$. Then the continuity of $\mu$ implies that $\mu(E)=0$, which is a contradiction. Thus $h(K)=[0,1]$.

We prove that $Y=\{y\in [0,1]: \#h^{-1}(y)\geq 2\}$ is countable. For all $n\in \NN^+$ let
$$Y_n=\{y\in [0,1]: \diam h^{-1}(y)\geq 2/n\}.$$
As $Y=\bigcup_{n=1}^{\infty} Y_n$, it is enough to show that $Y_n$ is finite for all
$n\in \NN^+$. Let us fix $n$ and for all $y\in Y_n$ pick $a_y,b_y\in h^{-1}(y)$ such that $a_y\prec b_y$ and $d(a_y,b_y)>1/n$. Let us say that a set is $1/n$-separated if the distance between every pair of its points is at least $1/n$. Since a compact metric space does not have an infinite $1/n$-separated subspace, it is enough to prove that $A_n=\{a_y: y\in Y_n\}$ is $1/n$-separated. Let $y,w\in Y_n$ such that $y<w$. Then the definition of $h$ yields $a_y\prec b_y\preceq a_w$. Thus the definition of the order $\prec$ implies
$$1/n<d(a_y,b_y)=\diam [a_y,b_y]\leq \diam [a_y,a_w]=d(a_y,a_w),$$
so $A_n$ is $1/n$-separated.

Let us define
\begin{align*} X&=\left\{x\in K: \limsup_{r\to 0+} \frac{\mu(B(x,r))}{r^{s}}<\infty \textrm{ for all } s<\dim_H \mu\right\}, \\
Z&=\left\{x\in K: \liminf_{r\to 0+} \frac{\mu(B(x,r))}{r^{s}}<\infty \textrm{ for all } s<\dim_P \mu\right\}.
\end{align*}
Theorem~\ref{t:limsup} yields that $\mu(X\cap Z)=1$. Since $Y$ is countable and $\mu$ is continuous, \eqref{eq:mert} implies that $\mu(h^{-1}(Y))=0$.
Therefore we can choose compact sets $K_n\subset (X\cap Z)\setminus h^{-1}(Y)$ such that $\mu(K_n)>0$ for all $n\in \NN^+$ and
$\mu(\bigcup_{n=1}^{\infty} K_n)=1$. Clearly, the $h_n$ are one-to-one, so \eqref{e2} holds. Then \eqref{eq:mert} yields \eqref{e3} and $\lambda(D_n)>0$, thus \eqref{e1} is satisfied.

Now we prove \eqref{e4}. We may assume that $\dim_H \mu>0$, otherwise we are done.
Let us fix $0<s<\dim_H \mu$. As $K_n\subset X$ for all $n\in \NN^+$, it is enough to show that for each
non-empty $A\subset X$ we have
\begin{equation}\label{eq:HhA} \dim_H h(A)\leq \frac{\dim_H A}{s},
\end{equation}
then letting $s\nearrow \dim_H \mu$ yields \eqref{e4}.
Fix a non-empty $A\subset X$ and for all $i\in \NN^+$ let
$$X_i=\{x\in X: \mu(B(x,r))\leq ir^s \textrm{ for all } r\geq 0\}.$$
The definition of $X$ clearly implies that $X=\bigcup_{i=1}^{\infty} X_i$.
The definitions of $h$ and the linear order $\prec$
yield that for all $x,z\in X_i$ with $x \preceq z$ we have
$$|h(x)-h(z)|=\mu([x,z))\leq \mu(B(x,d(x,z)))\leq i(d(x,z))^s,$$
so the $h|_{X_i}$ are $s$-H\"older.
By Fact~\ref{f:Holder} we have $\dim_H h(A\cap X_i)\leq (\dim_H (A\cap X_i))/s$ for all $i\in \NN^+$.
Therefore $A\subset X=\bigcup_{i=1}^{\infty} X_i$ and the countable stability of the Hausdorff dimension imply that
\begin{align*} \dim_H h(A)=\sup_{i\in \NN^+} \dim_H h(A\cap X_i)\leq \sup_{i\in \NN^+} \frac{\dim_H (A\cap X_i)}{s}=\frac{\dim_H A}{s},
\end{align*}
thus \eqref{eq:HhA} holds. Hence \eqref{e4} is satisfied.

Finally, we show \eqref{e5}. We may assume that $\dim_P \mu>0$, otherwise we are done.
Let us fix an arbitrary $0<s<\dim_P \mu$. As $K_n\subset Z$ for all $n\in \NN^+$,
it is enough to show that for each fixed
non-empty $A\subset Z$ we have $\dim_H h(A)\leq (\dim_P A)/s$,
then letting $s\nearrow \dim_P \mu$ yields \eqref{e5}. We may suppose that $\dim_P A<\infty$, otherwise we are done.
It is enough to show that for each fixed $t>\dim_P A$ we have
\begin{equation} \label{eq:Hha2} \dim_H h(A)\leq t/s, \end{equation}
then letting $t\searrow \dim_P A$ finishes the proof. The definition of the packing dimension implies that
there are sets $A_i$ such that $A=\bigcup_{i=1}^{\infty} A_i$ and $\overline{\dim}_{B} A_i<t$ for all $i\in \NN^+$.
For all $j\in \NN^+$ let
$$Z_j=\{x\in Z: \mu(B(x,2^{-n}))\leq j2^{-ns} \textrm{ for infinitely many } n\in \NN^+\}.$$
The definition of $Z$ implies that $Z=\bigcup_{j=1}^{\infty} Z_j$.  Let us fix
$i,j\in \NN^+$ and let $D=A_i\cap Z_j$. Now we show that
\begin{equation} \label{eq:Hha3} \dim_H h(D)\leq t/s.
\end{equation}
Since $\overline{\dim}_{B} D<t$, we can fix $u$ such that $\overline{\dim}_{B} D<u<t$. For all $n\in \NN^+$ let
\begin{align*} \iN_n&=\{B(x,2^{-n}): x\in D\}, \\
\iS_n&=\{B(x,2^{-n}): x\in D \textrm{ and } \mu(B(x,2^{-n}))\leq j2^{-ns}\}.
\end{align*}
Then clearly $\iS_n\subset \iN_n$. Fact~\ref{f:ultra} and the compactness of $K$ yield that $\iN_n$ consists of finitely many pairwise disjoint balls, so $\#\iN_n=N_{2^{-n}}(D)$ for all $n\in \NN^+$, where we recall that $N_{2^{-n}}(D)$
is the smallest number of closed balls with radius $2^{-n}$ whose union cover $D$.
Thus $\overline{\dim}_{B} D<u$ and the definition of the
upper box dimension yield that for all $n\in \NN^+$ we have
\begin{equation} \label{eq:n1} \#\iS_n\leq N_{2^{-n}}(D)\leq c_1 2^{nu},
\end{equation}
where $c_1\in \RR^+$.
Let $S\in \iS_n$ for some $n\in \NN^+$. For all $x,z\in S$ with $x \preceq z$ the definition of $\prec$ implies
that $[x,z)\subset B(x,d(x,z))\subset B(x,2^{-n})$. Thus the definition of $h$ and $x\in Z_j$ yield
$$|h(x)-h(z)|=\mu([x,z))\leq \mu(B(x,2^{-n}))\leq c_2 2^{-ns},$$
where $c_2=j$. Therefore for all $n\in \NN^+$ and $S\in \iS_n$ we obtain that
\begin{equation} \label{eq:n2}  \diam h(S) \leq c_2 2^{-ns}.
\end{equation}
Since $D\subset Z_j$, we have $D\subset \bigcup_{n=N}^{\infty} \bigcup_{S\in \iS_n} S$ for all $N\in \NN^+$.
Thus for all $N\in \NN^+$ we have
\begin{equation} \label{eq:n3} h(D)\subset \bigcup_{n=N}^{\infty} \bigcup_{S\in \iS_n} h(S).
\end{equation}
For all $N\in \NN^+$ let $\delta_N=j2^{-Ns}$. Then \eqref{eq:n3}, \eqref{eq:n1}, \eqref{eq:n2} and $u<t$ imply
\begin{align*} \iH^{t/s}_{\delta_N}(h(D))&\leq  \sum_{n=N}^{\infty} \sum_{S\in \iS_n} (\diam h(S))^{t/s} \\
&\leq \sum_{n=N}^{\infty} c_1 2^{nu} (c_2 2^{-ns})^{t/s}=c_3 2^{N(u-t)},
\end{align*}
where $c_3=c_1 {c_2}^{t/s}(1-2^{u-t})^{-1}$. Thus $u<t$ yields that $\lim_{N\to \infty} \iH^{t/s}_{\delta_N}(h(D))=0$, so
$\iH^{t/s}(h(D))=0$. Therefore $\dim_H h(D)\leq t/s$, thus \eqref{eq:Hha3} holds.

As $A=\bigcup_{i=1}^{\infty} \bigcup_{j=1}^{\infty} (A_i\cap Z_j)$,
the countable stability of the Hausdorff dimension and \eqref{eq:Hha3} imply
$$\dim_H h(A)=\sup_{i,j\in \NN^+} \dim_H h(A_i\cap Z_j)\leq t/s,$$
thus \eqref{eq:Hha2} holds. This implies \eqref{e5}, and the proof is complete.
\end{proof}

\begin{theorem} \label{t:mainu}
Let $K$ be a compact ultrametric space, and let $d\in \NN^+$. If $\mu$ is a continuous mass distribution on $K$, then for the prevalent $f\in C(K,\RR^d)$ there is an open set $U_{f}\subset \RR^d$ such that $\mu(f^{-1}(U_f))=\mu(K)$ and for all $y\in U_{f}$ we have
$$\dim_H f^{-1}(y)\geq \dim_H \mu \quad \textrm{and} \quad \dim_P f^{-1}(y)\geq \dim_P \mu.$$
\end{theorem}

\begin{proof}
For all $n\in \NN^+$ let us choose compact sets $K_n\subset K$ and $D_n\subset \RR$ and homeomorphisms
$h_n\colon K_n\to D_n$ according to Lemma \ref{l:Holder}. As $\mu(\bigcup_{n=1}^{\infty} K_n)=\mu(K)$,
Lemma~\ref{l:occup} yields that it is enough to prove that the sets
\begin{align*} \iA_n=\{&f\in C(K_n,\RR^d): \exists \textrm{ an open set } U_{f}\subset \RR^d \textrm{ such that } \mu(f^{-1}(U_f))=\mu(K_n) \\
&\textrm{and }\dim_H f^{-1}(y)\geq \dim_H \mu \textrm{ and } \dim_P f^{-1}(y)\geq \dim_P \mu \textrm{ for all } y\in U_{f}\}
\end{align*}
are prevalent in $ C(K_n,\RR^d)$. As $\lambda(D_n)>0$ by \eqref{e1}, Theorem~\ref{t:real} implies that
\begin{align*} \iB_n=\{&f\in C(D_n,\RR^d): \exists \textrm{ an open set } U_{f}\subset \RR^d \textrm{ such that } \\
&\lambda(f^{-1}(U_f))=\lambda(D_n) \textrm{ and } \dim_H f^{-1}(y)=1 \textrm{ for all } y\in U_{f}\}
\end{align*}
are prevalent in $ C(D_n,\RR^d)$. Fix $n\in \NN^+$ and define
$$H_n\colon C(K_n,\RR^d)\to C(D_n,\RR^d), \quad  H_n(f)=f\circ h_n^{-1}.$$
Then $H_n$ is a continuous group isomorphism, so Lemma~\ref{l:D} yields that $H_n^{-1}(\iB_n)$ is prevalent in $C(K_n,\RR^d)$.
Therefore it is enough to prove that $H_n^{-1}(\iB_n)\subset \iA_n$. Let us fix $f\in H_n^{-1}(\iB_n)$, we need to prove that $f\in \iA_n$.
Let $g=H_n(f)=f\circ h_n^{-1} \in \iB_n$, then there exists a non-empty open set $U_{g}\subset \RR^d$ such that $\lambda(g^{-1}(U_g))=\lambda(D_n)$ and
$\dim_H g^{-1}(y)=1$ for all $y\in U_{g}$. Let $U_{f}=U_{g}$. Applying \eqref{e3} twice with $K_n=h_n^{-1}(D_n)$ implies that
\begin{equation*} \label{eq:mugg} \mu(f^{-1}(U_f))=\mu(h_n^{-1}(g^{-1}(U_g)))=\lambda(g^{-1}(U_g))=\lambda(D_n)=\mu(K_n).
\end{equation*}
Let us fix $y\in U_{f}$. Then \eqref{e4} and \eqref{e5} yield that
\begin{align*} \label{twoeq}
\dim_H f^{-1}(y)&= \dim_H h_n^{-1} (g^{-1}(y))\geq \dim_H \mu \cdot \dim_H g^{-1} (y)=\dim_H \mu, \\
\dim_P f^{-1}(y)&= \dim_P h_n^{-1} (g^{-1}(y))\geq \dim_P \mu \cdot \dim_H g^{-1} (y)=\dim_P \mu.
\end{align*}
These imply that $f\in \iA_n$, and the proof is complete.
\end{proof}

\subsection{The Main Theorem} \label{ss:main} We prove the Main Theorem after some preparation.

\begin{definition}\label{d:nonex} Let $X$ be a metric space. For all $r>0$ let $N(r)$ be the minimal number such that every closed
ball $B(x,r)$ can be covered by $N(r)$ closed balls of radius $r/2$. Then $X$ is called \emph{doubling} if $\sup\{N(r): r>0\}<\infty$. We say that $X$ is \emph{non-exploding} if
$$\lim_{r\to 0+} \frac{\log N(r)}{\log r}=0.$$
\end{definition}

Every subspace of $\RR^m$ is doubling, and every doubling space is non-exploding.

\begin{definition} Let $X,Y$ be metric spaces. A map $f\colon X\to Y$ is called \emph{nearly Lipschitz} if $f$ is $s$-H\"older for all $s<1$. We say that
$f$ is \emph{nearly bi-Lipschitz} if it is one-to-one and both $f$ and $f^{-1}$ are nearly Lipschitz. The spaces $X$ and $Y$ are
\emph{nearly Lipschitz equivalent} if there exists a nearly bi-Lipschitz onto map $f\colon X\to Y$.
\end{definition}

\begin{fact} \label{f:nH} If $X,Y$ are metric spaces and $f\colon X\to Y$ is a nearly bi-Lipschitz map
then $\dim_H f(A)=\dim_H A$ and $\dim_P f(A)=\dim_P A$ for all $A\subset X$.
\end{fact}

The above fact follows from Fact~\ref{f:Holder}. For the following theorem see \cite{Z}.

\begin{theorem}[Zindulka] \label{t:Z} Let $K$ be a non-exploding compact metric space. If $\mu$ is a mass distribution on $K$,
then there exists a compact set $C\subset K$ such that $\mu(C)>0$ and $C$ is nearly bi-Lipschitz equivalent to an ultrametric space.
\end{theorem}

\begin{theorem}[Main Theorem] \label{t:Main} Let $K$ be a non-exploding compact metric space, and let $d\in \NN^+$.
If $\mu$ is a continuous mass distribution on $K$, then for the prevalent $f\in C(K,\RR^d)$ there is an open set $U_f\subset \RR^d$ such that $\mu(f^{-1}(U_f))=\mu(K)$ and for all $y\in U_f$ we have
$$\dim_H f^{-1}(y)\geq \dim_H \mu \quad \textrm{and} \quad \dim_P f^{-1}(y)\geq \dim_P \mu.$$
\end{theorem}

In the special case $K=[0,1]^m$ we obtain the following:

\begin{corollary} \label{c:occup2}  Let $m,d\in \NN^+$.
Then for the prevalent $f\in C([0,1]^m,\RR^{d})$
there is an open set $U_{f}\subset \RR^d$ such that
$\lambda^m(f^{-1}(U_f))=1$ (hence $U_f$ is dense in $f([0,1]^m)$ and for all $y\in U_{f}$ we have
we have
$$\dim_{H} f^{-1} (y)=m.$$
\end{corollary}

\begin{proof}[Proof of the Main Theorem]
By Theorem~\ref{t:Z} there exist compact sets $K_n\subset K$ such that $\mu(K_n)>0$ for all
$n\in \NN^+$ and $\mu(\bigcup_{n=1}^{\infty}K_n)=\mu(K)$,
and there are compact ultrametric spaces $C_n$ and nearly bi-Lipschitz onto maps
$h_n\colon K_n\to C_n$. Define the measures $\mu_n=\mu|_{K_n}$ on $K_n$ and $\nu_n=\mu\circ h_n^{-1}$ on $C_n$ for all $n\in \NN^+$.
Since $\nu_n(C_n)=\mu(K_n)>0$, the measures $\mu_n$ and $\nu_n$ are continuous mass distributions.
Then $\dim_H \mu_n\geq \dim_H \mu$ and $\dim_P \mu_n\geq \dim_P \mu$ by the definition,
thus Lemma~\ref{l:occup} yields that it is enough to prove that the sets
\begin{align*} \iA_n=\{&f\in C(K_n,\RR^d): \exists \textrm{ an open set } U_{f}\subset \RR^d \textrm{ such that } \mu_n(f^{-1}(U_f))=\mu_n(K_n) \\
&\textrm{and } \dim_H f^{-1}(y)\geq \dim_H \mu_n \textrm{ and } \dim_P f^{-1}(y)\geq \dim_P \mu_n \textrm{ for all } y\in U_{f}\}
\end{align*}
are prevalent in $C(K_n,\RR^d)$. Theorem~\ref{t:mainu} implies that the sets
\begin{align*} \iB_n=\{&f\in C(C_n,\RR^d): \exists \textrm{ an open set } U_{f}\subset \RR^d \textrm{ such that } \nu_n(f^{-1}(U_f))=\nu_n(C_n) \\
&\textrm{and } \dim_H f^{-1}(y)\geq \dim_H \nu_n \textrm{ and } \dim_P f^{-1}(y)\geq \dim_P \nu_n\textrm{ for all } y\in U_{f}\}
\end{align*}
are prevalent in $C(C_n,\RR^d)$. Fix $n\in \NN^+$ and define $H_n\colon C(K_n,\RR^d)\to C(C_n,\RR^d)$ as $H_n(f)=f\circ h_n^{-1}$.
Then $H_n$ is a continuous isomorphism, so Lemma~\ref{l:D} yields that $H_n^{-1}(\iB_n)$ is prevalent in $C(K_n,\RR^d)$.
Since $h_n$ is nearly bi-Lipschitz, by Fact~\ref{f:nH} we have $\dim_H h_n^{-1}(B)=\dim_H B$ and
$\dim_P h_n^{-1}(B)=\dim_P B$ for all $B\subset C_n$.
Hence $\mu\circ h_n^{-1}=\nu_n$ yields that $\dim_H \mu_n=\dim_H \nu_n$ and $\dim_P \mu_n=\dim_P \nu_n$.
These easily imply that $H_n^{-1}(\iB_n)=\iA_n$, so $\iA_n$ is prevalent. The proof is complete.
\end{proof}

\begin{corollary}\label{c:Mainp}
Let $K$ be an uncountable, non-exploding compact metric space and let $d\in \NN^+$. Then for the prevalent
$f\in C(K,\RR^d)$ for all $s<\dim_P K$ there is a non-empty open set $U_{f,s}\subset \RR^d$ such that for all $y\in U_{f,s}$
we have
$$\dim_P f^{-1} (y)\geq s.$$
In particular, we have
$$\sup\{\dim_P f^{-1}(y): y\in \RR^d\}=\dim_P K.$$
\end{corollary}

\begin{proof} Fix $s=0$ if $\dim_P K=0$ and $s\in (0,\dim_P K)$ if $\dim_P K>0$.
Theorem~\ref{t:frostman} implies that there is a continuous mass distribution
$\mu$ on $K$ with $\dim_P \mu\geq s$. Applying the Main Theorem for $\mu$ yields that
\begin{align*} \iA(s)=\{&f\in C(K,\RR^d): \exists \textrm{ a non-empty open set } U_{f}\subset \RR^d \\
&\textrm{such that } \dim_P f^{-1}(y)\geq s \textrm{ for all } y\in U_{f}\}
\end{align*}
is prevalent in $C(K,\RR^d)$. If $\dim_P K=0$ then we are done, otherwise
choose a sequence $s_n\nearrow \dim_P K$, then $\bigcap_{n=1}^{\infty} \iA(s_n)$ is the desired prevalent set.
\end{proof}

In the case of Hausdorff dimension we generalize the above two corollaries.
For the following theorem see \cite[Theorem~1.4]{MN}.

\begin{theorem}[Mendel-Naor] \label{t:Naor} If $K$ is a compact metric space and $s<\dim_H K$ then
there exists a compact set $C\subset K$ such that $\dim_H C>s$ and $C$ is bi-Lipschitz equivalent to an ultrametric space.
\end{theorem}

\begin{theorem}\label{t:main}
Let $K$ be an uncountable compact metric space and let $d\in \NN^+$. Then for the prevalent
$f\in C(K,\RR^d)$ for all $s<\dim_H K$ there is a non-empty open set $U_{f,s}\subset \RR^d$ such that for all $y\in U_{f,s}$
we have
$$\dim_H f^{-1}(y)\geq s.$$
In particular, we have
$$\sup\{\dim_H f^{-1}(y): y\in \RR^d\}=\dim_H K.$$
\end{theorem}

\begin{proof} By Theorem~\ref{t:D} we may assume $\dim_H K>0$. Fix $s\in(0,\dim_H K)$ and let
\begin{align*} \iA=\{&f\in C(K,\RR^d): \exists \textrm{ a non-empty open set } U_{f}\subset \RR^d \\
&\textrm{such that } \dim_H f^{-1}(y)\geq s \textrm{ for all } y\in U_{f}\}.
\end{align*}
It is enough to prove that $\iA=\iA(s)$ is prevalent, since we can choose a sequence $s_n\nearrow\dim_H K$ and
$\bigcap_{n=1}^{\infty} \iA(s_n)$ will be a prevalent set in $C(K,\RR^d)$ satisfying the theorem.
By Theorem~\ref{t:Naor} there is a compact set $C\subset K$ such that $\dim_H C>s$ and there exist a compact ultrametric space
$D$ and a bi-Lipschitz onto map $h\colon C\to D$. By Fact~\ref{f:nH} we have $\dim_H D>s$,
so Theorem~\ref{t:frostman} yields that there exists a continuous mass distribution $\mu$ on $D$ such that $\dim_H \mu\geq s$.
Therefore Theorem~\ref{t:mainu} implies that
\begin{align*}\iB=\{&g\in C(D,\RR^d): \exists \textrm{ a non-empty open set } U_{g}\subset \RR^d \\
&\textrm{such that } \dim_H g^{-1}(y)\geq s  \textrm{ for all } y\in U_{g}\}
\end{align*}
is prevalent in $C(D,\RR^d)$. Consider
\begin{equation*} R\colon C(K,\RR^d)\to C(D,\RR^d), \quad R(f)=f|_{C}\circ h^{-1}.
\end{equation*}
As $h$ is a homeomorphism, $R$ is a composition of two continuous onto homomorphisms, so
$R$ is a continuous onto homomorphism. Thus Lemma~\ref{l:D} implies that $R^{-1}(\iB)\subset C(K,\RR^d)$ is prevalent, so it is enough to prove that $R^{-1}(\iB)\subset \iA$. Let us fix $f\in R^{-1}(\iB)$, we need to prove that $f\in \iA$.
Let $g=R(f)=f|_{C}\circ h^{-1}\in \iB$, then there exists a non-empty open set $U_{g}\subset \RR^d$
such that $\dim_H g^{-1}(y)\geq s$ for all $y\in U_{g}$. Let $U_{f}=U_{g}$ and fix $y\in U_{f}$.
Then $h^{-1}(g^{-1}(y))=(f|_{C})^{-1}(y)\subset f^{-1}(y)$,
so applying Fact~\ref{f:nH} for the bi-Lipschitz map $h$ yields that
$$\dim_H f^{-1}(y)\geq \dim_H h^{-1} (g^{-1}(y))=\dim_H g^{-1} (y)\geq s.$$
Hence $f\in \iA$, and the proof is complete.
\end{proof}

\subsection{Fibers of maximal dimension} \label{ss:max} First we prove that one cannot replace supremum with
maximum in the second claims of Corollary~\ref{c:Mainp} and Theorem~\ref{t:main}. For the following well-known lemmas see e.g. \cite[Lemma~4]{Z} and
\cite{BDE}, respectively. Note that Lemma~\ref{l:iK} is stated in \cite{BDE} only in the case $K=[0,1]$,
but the proof works verbatim for all $K$.

\begin{lemma} \label{l:nshy} Let $G$ be an abelian Polish group and let $A\subset G$. If for all compact set $K\subset G$ there exists a $g\in G$ such that $K+g\subset A$ then $A$ is non-shy.
 \end{lemma}

\begin{lemma} \label{l:iK} Let $K\subset [0,1]$ and $\iK\subset C(K,\RR)$ be compact sets. Then there is a strictly increasing function $h\in C[0,1]$
such that $h(0)=0$ and for all $f\in \iK$ and $x,z\in K$, $x\neq z$ we have
$$|f(x)-f(z)|<h(|x-z|).$$
\end{lemma}

\begin{theorem} \label{t:ex} There is a compact set $K\subset \RR$ such that $\dim_H K = \dim_P K = 1$ and
$$\iA = \{f\in C(K,\RR): \dim_H f^{-1}(y) \le \dim_P f^{-1}(y)<1 \textrm{ for all } y\in \RR\}$$
is non-shy in $C(K,\RR)$.
\end{theorem}

\begin{proof} For all $n\in \NN^+$ let $K_n\subset [0,1/n]$ be compact sets such that $\dim_H K_n=\dim_P K_n=1-1/n$ and let $K=\bigcup_{n=1}^{\infty} K_n \cup \{0\}$. Then $K\subset [0,1]$ is compact and $\dim_H K=\dim_P K=1$. Define
$$\iB=\{f\in C(K,\RR): f^{-1}(f(0))=\{0\}\}.$$
Now we show that $\iB\subset \iA$, that is, $\dim_H f^{-1}(y) \le \dim_P f^{-1}(y)<1$ for every $f\in \iB$ and $y\in \RR$.
The first inequality clearly holds, so it is enough to prove that $\dim_P f^{-1}(y)<1$.
Let us fix $f\in \iB$, by definition $f^{-1}(f(0))=\{0\}$. For all
$y\in \RR\setminus \{f(0)\}$ the level set $f^{-1}(y)\subset K\setminus \{0\}$ is compact,
thus it can be covered by finitely many sets $K_n$. Therefore the
countable stability of packing dimension yields that $\dim_P f^{-1}(y)<1$.

Finally, it is enough to show that $\iB$ is non-shy in $C(K,\RR)$. Let $\iK\subset C(K,\RR)$ be an arbitrary compact set,
by Lemma~\ref{l:nshy} it is enough to prove that there exists a $g\in C(K,\RR)$ with
$\iK+g\subset \iB$. By Lemma~\ref{l:iK}
there is a strictly increasing function $h\in C[0,1]$ such that $h(0)=0$ and $|f(x)-f(0)|<h(x)$
for all $f\in \iK$ and $x\in K\setminus \{0\}$. Let $g=h|_{K}$, then for all $f\in \iK$ and $x\in K\setminus \{0\}$ we have
$$f(x)+g(x)=f(x)+h(x)>f(0)=f(0)+h(0)=f(0)+g(0),$$
thus $f+g\in \iB$. This completes the proof.
\end{proof}

The Main Theorem easily implies that if $K$ is `large in its dimension' then Corollary~\ref{c:Mainp} and
Theorem~\ref{t:main} can be strengthened as follows.

\begin{corollary} \label{c:MainHP} Let $d\in \NN^+$ and let $\dim$ be one of $\dim_H$ or $\dim_P$.
Assume that $K$ is a non-exploding compact metric space and there is a continuous mass distribution $\mu$ on $K$ such that
$\dim \mu=\dim K$. Then for the prevalent $f\in C(K,\RR^d)$ there is an open set $U_f\subset \RR^d$ such that
$\mu(f^{-1}(U_f))=\mu(K)$ and for all $y\in U_f$ we have
$$\dim f^{-1}(y)=\dim K.$$
\end{corollary}

\begin{remark}
Note that the compact set $K$ in Theorem~\ref{t:ex} can be decomposed as $K=\bigcup_{n=1}^{\infty} A_n$ such that $\dim_H A_n<\dim_H K$.
Assume that $K$ is a non-exploding compact metric space such that $\dim_H K<\infty$
and such a decomposition does not exist. We sketch that there is a mass distribution $\mu$ on $K$ with $\dim_H \mu=\dim_H K$,
so Corollary~\ref{c:MainHP} holds for $K$ in the case of Hausdorff dimension. Indeed, \cite[Theorem~2]{R}
(see also the more general \cite[Theorem~6.4]{SS}) implies that there is a gauge
function (see Section~\ref{s:gauge} for the definition) $\varphi \colon [0,\infty)\to [0,\infty)$ such that
$$\liminf_{r\to 0+} \frac{\log \varphi(r)}{\log r}=\dim_H K \quad \textrm{and} \quad \iH^{\varphi}(K)>0,$$
where $\iH^{\varphi}$ denotes the $\varphi$-Hausdorff measure (see Section~\ref{s:gauge} for the definition).
As $\dim_H K<\infty$, we may assume that $\varphi$ is of finite order,
that is, $\varphi(2r)\leq c\varphi(r)$ for some $c\in \RR^+$ and for all $r\in [0,\infty)$. By \cite{Ho} there is a
compact set $C\subset K$ such that $0<\iH^{\varphi}(C)<\infty$. Then
$\mu=\iH^{\varphi}|_{C}$ is a mass distribution on $K$ with $\dim_H \mu=\dim_H K$.
\end{remark}

\subsection{The homogeneous case} \label{ss:hom}

Let us now consider continuous mass distributions $\mu$ on $K$ such that $\supp \mu=K$.
Then the larger $\dim_H \mu$ or $\dim_P \mu$ can be, the more homogeneous $K$ is. The Main Theorem yields the following corollary.

\begin{corollary} \label{c:Main2} Let $d\in \NN^+$ and let $K$ be a non-exploding compact metric space. Assume that
there is a continuous mass distribution $\mu$ on $K$ such that $\supp \mu=K$. Then for the prevalent $f\in C(K,\RR^d)$ there exists an open set
$U_f\subset \RR^d$ such that $\mu(f^{-1}(U_f))=\mu(K)$ (hence $U_f$ is dense in $f(K)$) and for all $y\in U_f$
we have
$$\dim_H f^{-1}(y)\geq \dim_H \mu \quad \textrm{and} \quad \dim_P f^{-1}(y)\geq \dim_P \mu.$$
\end{corollary}

If $K\subset \RR^m$ is a self-similar set satisfying the open set condition, we can say more. Recall that
$\iP^s$ denotes the $s$-dimensional packing measure.

\begin{corollary} \label{c:HP} Let $m,d\in \NN^+$ and let $K\subset \RR^m$ be a self-similar set satisfying the open set
condition. It is well-known that $\dim_H K=\dim_P K=s$ and $\iH^s(K),\iP^s(K)\in \RR^+$.
Then for the prevalent $f\in C(K,\RR^d)$ there exists an open set $U_f\subset \RR^d$ such that
$\iH^s(f^{-1}(U_f))=\iH^s(K)$ (hence $U_f$ is dense in $f(K)$) and
$$\dim_H f^{-1}(y)=s \textrm{ for all } y\in U_f.$$
Similarly, for the prevalent $f\in C(K,\RR^d)$ there exists an open set $V_f\subset \RR^d$ such that
$\iP^s(f^{-1}(V_f))=\iP^s(K)$ (hence $V_f$ is dense in $f(K)$) and
$$\dim_P f^{-1}(y)=s \textrm{ for all } y\in V_f.$$
\end{corollary}

\begin{proof} For $\dim_H K=\dim_P K=s$ and $\iH^s(K),\iP^s(K)\in \RR^+$ see \cite[Theorem~2.7]{F2}. Applying the
Main Theorem for the mass distributions $\mu=\iH^s$ and $\mu=\iP^s$ concludes the proof, we need to
show only that $U_f$ and $V_f$ are dense in $f(K)$. Let $U$ be a non-empty open set in $K$, then it is enough to
prove that $\iH^s(U)>0$ and $\iP^s(U)>0$. Since $K$ is self-similar, $U$ contains a similar copy of $K$
with some similarity ratio $r>0$.
Then the scaling property of the Hausdorff and the packing measures yields that $\iH^s(U)\geq r^s\iH^s(K)>0$ and $\iP^s(U)\geq r^s\iP^s(K)>0$.
\end{proof}

Finally, we prove two characterization theorems.

\begin{theorem} \label{t:chr} If $K$ is a compact metric space and $d\in \NN^+$ then the following
statements are equivalent:
\begin{enumerate}[(i)]
\item  $\dim_{H} f^{-1}(y)=\dim_{H}K$ for the prevalent $f\in C(K,\RR^d)$ and for the generic $y\in f(K)$;
\item  $\dim_H U=\dim_H K$ for every non-empty open set $U\subset K$.
\end{enumerate}
\end{theorem}

\begin{proof}
$(ii) \Rightarrow (i)$: We may assume that $\dim_H K>0$, otherwise the statement is trivial.
Choose a positive sequence $s_i\nearrow \dim_H K$ and let $\iV=\{V_n: n\in \NN^+\}$ be a countable basis of $K$ consisting of non-empty open sets.
For all $n\in \NN^+$ let $K_n=\cl V_n$, then $\dim_H K_n=\dim_H K$. For all $i,n\in \NN^+$ consider
\begin{align*} \iA_{i,n}=\{&f\in C(K_n,\RR^d): \exists \textrm{ a non-empty open set } U_{f,s_i}\subset \RR^d \\
&\textrm{such that } \dim_H f^{-1}(y)\geq s_i \textrm{ for all } y\in U_{f,s_i}\}.
\end{align*}
Theorem~\ref{t:main} implies that the $\iA_{i,n}$ are prevalent. For all $n\in \NN^+$ let us define
$$R_n\colon C(K,\RR^d)\to C(K_n,\RR^d), \quad R_n(f)=f|_{K_n}.$$
Corollary~\ref{c:hereditary} yields that the $R^{-1}_n(\iA_{i,n})$ are prevalent in $C(K,\RR^d)$ for all $i,n\in \NN^+$.
As a countable intersection of prevalent sets,
$\iA=\bigcap_{i=1}^{\infty} \bigcap_{n=1}^{\infty} R^{-1}_n(\iA_{i,n})$ is also prevalent in $C(K,\RR^d)$. For all $f\in \iA$ let
$$U_f=\bigcap_{i=1}^{\infty} \left(\bigcup_{n=1}^{\infty} U_{f|_{K_n},s_i}\right).$$
As a countable intersection of dense open sets, $U_f$ is co-meager in $f(K)$. Let us fix $f\in \iA$ and $y\in U_f$, it is enough to prove
that $\dim_H f^{-1}(y)=\dim_H K$. For all $i\in \NN^+$ there is an $n\in \NN^+$ such that $y\in U_{f|_{K_n},s_i}$, therefore
$$\dim_H f^{-1}(y)\geq \dim_H (f|_{K_n})^{-1}(y)\geq s_i.$$
As $s_i\nearrow \dim_H K$, we obtain that $\dim_H f^{-1}(y)=\dim_H K$. Hence $(i)$ holds.

$(i) \Rightarrow (ii)$: Assume to the contrary that there exist $x\in K$ and $r>0$ such that $\dim_H U(x,r)<\dim_H K$. Tietze's extension theorem implies that there is a $g\in C(K,\RR^d)$ such that $g(K\setminus U(x,r))$ and $g(B(x,r/2))$ are distinct points in $\RR^d$. Then there exist an $\varepsilon>0$ and an
open set $U\subset \RR^d$ such that $f(B(x,r/2))\subset U$ and $f(K\setminus U(x,r))\subset \RR^d\setminus U$ for all $f\in B(g,\varepsilon)$. Clearly, $B(g,\varepsilon)$ is non-shy in $C(K,\RR^d)$. If  $f\in B(g,\varepsilon)$ then $U\cap f(K)$ is a non-empty open set in $f(K)$ such that for every
$y\in U$ we have $\dim_H f^{-1}(y)\leq \dim_H U(x,r)<\dim_H K$, which contradicts $(i)$.
\end{proof}

\begin{theorem} If $K$ is a non-exploding compact metric space and $d\in \NN^+$ then the following
statements are equivalent:
\begin{enumerate}[(i)]
\item $\dim_{P} f^{-1}(y)=\dim_{P}K$ for the prevalent $f\in C(K,\RR^d)$ and for the generic $y\in f(K)$;
\item  $\dim_P U=\dim_P K$ for every non-empty open set $U\subset K$.
\end{enumerate}
\end{theorem}

\begin{proof} Repeat the proof of Theorem~\ref{t:chr}, only replace Hausdorff dimension with packing dimension, and apply Corollary~\ref{c:Mainp} instead of Theorem~\ref{t:main}.
\end{proof}

\section{Positively many level sets can be singletons}\label{s:sing}

In this section we only consider the space $C[0,1]$, our main goal is to prove Theorem~\ref{t:nonshy}.
The heart of the proof is the Theorem~\ref{t:sing}, which generalizes a result of Antunovi\'c et al.\
\cite[Proposition~3.3]{ABPR}, see also Theorem~\ref{t:ABPR}. First we need an easy lemma. Recall that $\iZ(f)=\{x\in [0,1]: f(x)=0\}$.

\begin{lemma} \label{l:Bor} The set $\iA=\{f\in C[0,1]: \iZ(f) \textrm{ is a singleton}\}$ is Borel.
\end{lemma}

\begin{proof} Clearly $\iA=\iB_1\setminus \iB_{2}$, where $\iB_i=\{f\in C[0,1]: \#\iZ(f)\geq i\}$. Since
$\iB_1$ is closed, it is enough to prove that $\iB_2$ is Borel. If $\iI$ denotes
the pairs of disjoint closed rational subintervals of $[0,1]$, then $\iB_2=\bigcup_{(I_1,I_2)\in \iI} \iB_{I_1,I_2}$, where
$$\iB_{I_1,I_2}=\{f\in C[0,1]: 0\in f(I_1)\cap f(I_2)\}.$$
Clearly $\iB_{I_1,I_2}$ are closed, therefore $\iB_2$ is $F_{\sigma}$, thus Borel.
\end{proof}

\begin{theorem} \label{t:sing} Let $\mu$ be a Borel probability measure on $C[0,1]$. Then there exists a function $g\in C[0,1]$ such that
$$\mu(\{f\in C[0,1]: \iZ(f-g) \textrm{ is a singleton}\})>0.$$
Consequently, the set $\iA=\{f\in C[0,1]: \iZ(f) \textrm{ is a singleton}\}$ is non-shy.
\end{theorem}

\begin{proof} It is enough to prove the first statement, because the definition of shyness readily implies the second one.
Lemma~\ref{l:Bor} implies that $\iA$ is Borel, so the set $\{f\in C[0,1]: \iZ(f-g) \textrm{ is a singleton}\}=\iA+g$ is also Borel for each
$g\in C[0,1]$.

By Theorem~\ref{t:Ulam}
we may assume by shifting, restricting and normalizing $\mu$ that there is a compact set $\iK\subset C[0,1]$ such that
$\mu(\iK)=1$ and $f(x)\in [0,1]$ for all $f\in \iK$ and $x\in [0,1]$. For each compact set $\Gamma \subset [0,1]^2$ let us
define the compact set
$$\pi(\Gamma)=\{f\in \iK: \exists (x,y)\in \Gamma \textrm{ such that } f(x)=y\}.$$
First assume that there exists a point $(x_0,y_0)\in [0,1]^2$ with $\mu(\pi(\{(x_0,y_0)\}))>0$. By Lemma~\ref{l:iK} there is a strictly increasing
function $h\in C[0,1]$ such that $h(0)=0$ and
for all $f\in \iK$ and $x,z\in [0,1]$, $x\neq z$ we have
\begin{equation} \label{eq:fh} |f(x)-f(z)|<h(|x-z|).
\end{equation}
Let us define $g\in C[0,1]$ as $g(x)=y_0+h(|x-x_0|)$. Then clearly $\iZ(f-g)=\{x_0\}$ for all $f\in \pi(\{(x_0,y_0)\})$, thus
$$\mu(\{f\in \iK: \iZ(f-g) \textrm{ is a singleton}\})\geq \mu(\pi(\{(x_0,y_0)\}))>0,$$
which concludes the proof.

Therefore we may assume that $\mu(\pi(\Delta))=0$ for every finite set $\Delta\subset [0,1]^2$.
We prove that it is enough to find a function $g\in C[0,1]$ such that
\begin{equation}\label{eq:isolated}
\mu(\{f\in \iK: \iZ(f-g) \textrm{ has an isolated point}\})>0.
\end{equation}
Indeed, by \eqref{eq:isolated} we may assume that there are rational numbers $0\leq a<b\leq 1$ with
\begin{equation} \label{eq:gafa} \mu(\{f\in \iK: \iZ(f-g)\cap [a,b] \textrm{ is a singleton},\, g(a)\leq f(a),\, f(b)\leq g(b)\})>0,
\end{equation}
since the other three cases with reversed inequalities are similar. Define $\widehat{g}\in C[0,1]$ as
$$\widehat{g}(x)=\begin{cases} g(a)-h(a-x) & \textrm{ if } x\in [0,a), \\
g(x) & \textrm{ if } x\in [a,b], \\
g(b)+h(x-b) & \textrm{ if } x\in(b,1].
\end{cases}$$
Then for all $f\in \iK$ we have $\widehat{g}(x)<f(x)$ if $x\in [0,a)$ and $f(x)<\widehat{g}(x)$ if $x\in (b,1]$. Thus
$\iZ(f-\widehat{g})=\iZ(f-g)\cap [a,b]$ for all $f\in \iK$, so \eqref{eq:gafa} implies that
$$\mu(\{f\in C[0,1]: \iZ(f-\widehat{g}) \textrm{ is a singleton}\})>0,$$
and we are done.

Now we show \eqref{eq:isolated}. First we define the function $g\in C[0,1]$. Let $\{\alpha_n\}_{n\in \NN}$ be a sequence of positive reals such that
\begin{equation} \label{eq:alfa} \alpha_0=1/2 \textrm{ and } \alpha_{n+1} \leq \alpha_{n}/2 \textrm{ for all } n\in \NN,
\end{equation}
the exact values will be given later. For all $n\in \NN^+$ and $(k_1,\dots, k_n)\in \{-1,1\}^n$ let
$$z_{k_1\dots k_n}=\frac12+\sum_{i=1}^n k_i \alpha_i.$$
Consider
$$Z=\{0\}\cup \left\{z_{k_1\dots k_n}: (k_1,\dots, k_n)\in \{-1,1\}^n, \, n\in \NN^+\right\}.$$
Let $g(0)=0$, and for all $n\in \NN^+$ and $(k_1,\dots, k_n)\in \{-1,1\}^n$ let
\begin{equation} \label{eq:gdef} g(z_{k_1\dots k_n})=\frac12+\sum_{i=1}^n \frac{k_i}{2^{i+1}}.
\end{equation}
For all $x\in [0,1]$ let
\begin{equation} \label{eq:gdef2} g(x)=\sup\{g(z): z\in Z, \, z\leq x\}.
\end{equation}
Then \eqref{eq:alfa} and \eqref{eq:gdef} easily imply that $g|_{Z}$ is well-defined and non-decreasing, so $g$ is also
well-defined and non-decreasing. Therefore the definitions yield
$g([0,1])\subset [g(0),g(1)]=[0,1]$. As a monotone function can have only jump discontinuities and
$g(Z)$ is dense in $[0,1]$ by \eqref{eq:gdef}, we obtain that $g\colon [0,1]\to [0,1]$ is continuous.

We prove that if the numbers $\alpha_n$ are small enough then $g$ satisfies \eqref{eq:isolated}.
For all $n\in \NN^+$ and $i\in \{1,\dots,2^n\}$ let $\phi(i,n)$ be the $i$th element of $\{-1,1\}^n$ with respect to the lexicographical ordering. Note that $\phi(i,n)$ precedes $\phi(j,n)$ with respect to this ordering iff $z_{\phi(i,n)} < z_{\phi(j,n)}$ with respect to the usual ordering of the real numbers.
Let $C_0=[0,1]$ and $\Gamma_0=\{1/2\}\times [0,1]$. For the definition of $h$ recall \eqref{eq:fh}.
Assume by induction that $\alpha_n$, $C_n$ and $\Gamma_n$ are already defined for some $n\in \NN$ and let
\begin{align}
C_{n+1}&=C_{n}\setminus \bigcup_{i=0}^{2^{n+1}} U\left(\frac{i}{2^{n+1}},h(4\alpha_{n+1})\right), \label{eq:cn}  \\
\Gamma_{n+1}&=\bigcup_{i=1}^{2^{n+1}} \left(\left\{z_{\phi(i,n+1)}\right\}\times \left(\left[\frac{i-1}{2^{n+1}},\frac{i}{2^{n+1}}\right]\cap C_{n+1} \right)\right) \label{eq:gamma}.
\end{align}
We show that if $\alpha_{n+1}\in (0,\alpha_n/2]$ is small enough then
\begin{equation} \label{eq:mucap} \mu(\pi(\Gamma_{n+1})\cap \pi(\Gamma_{n}))\geq \mu(\pi(\Gamma_{n}))-\frac{1}{2^{n+2}}.
\end{equation}
For all $m\in \NN^+$ let $C_{n+1}^{m}$ and $\Gamma_{n+1}^{m}$ be the sets $C_{n+1}$ and $\Gamma_{n+1}$ defined
by the value $\alpha_{n+1}=1/m$. For \eqref{eq:mucap} it is enough to prove that
\begin{equation} \label{eq:mun1} \lim_{m\to \infty} \mu(\pi(\Gamma_{n+1}^{m})\cap \pi(\Gamma_n))= \mu(\pi(\Gamma_n)). \end{equation}
Observe that $\Gamma_n$ consists of finitely many vertical line segments and let
$$\Delta_n=\{(x,y)\in \Gamma_{n}:2^{n+1}y\in \NN \textrm{ or } (x,y) \textrm{ is an endpoint of a segment of } \Gamma_n\}.$$
As $\Delta_n$ is finite, $\mu(\pi(\Delta_n))=0$ by our assumption. Since $h(4/m)\to 0$ as
$m\to \infty$, it is easy to see that for every given $(x,y)\in \Gamma_n\setminus \Delta_n$ there is
a constant $c=c(x,y)>0$ such that if $m$ is large
enough then $\Gamma_{n+1}^{m}$ contains either the vertical line segment $\{x-1/m\}\times [y-c,y+c]$ or
$\{x+1/m\}\times [y-c,y+c]$. Thus for every $f\in \pi(\Gamma_n)\setminus \pi(\Delta_n)$ there exists an
$M(f) \in \NN^+$ such that for all $m\geq M(f)$ we have $f\in \pi(\Gamma_{n+1}^{m})$.
For all $m\in \NN^+$ define
$$\iA_{m}=\{f\in (\pi(\Gamma_{n+1}^{m})\cap \pi(\Gamma_n))\setminus \pi(\Delta_n): M(f)\leq m\},$$
then clearly for all  $m\in \NN^+$ we have $\iA_m\subset \pi(\Gamma_{n+1}^m)\cap \pi(\Gamma_n)$.
 The definition of $M(f)$ implies that $\iA_m\subset \iA_{m+1}$
and the existence of $M(f)$ yields that $\bigcup_{m=1}^{\infty} \iA_m=\pi(\Gamma_n)\setminus \pi(\Delta_n)$.
Therefore $\mu(\pi(\Delta_n))=0$ implies that
$$\liminf_{m\to \infty}\mu(\pi(\Gamma_{n+1}^{m})\cap \pi(\Gamma_n))\geq \lim_{m\to \infty} \mu(\iA_m)=\mu(\pi(\Gamma_n)\setminus \pi(\Delta_n))=
\mu(\pi(\Gamma_n)),$$
so \eqref{eq:mun1} is satisfied. Thus $\alpha_n$, $C_n$ and $\Gamma_n$ can be defined for all $n\in \NN$ such that \eqref{eq:cn},
\eqref{eq:gamma} and \eqref{eq:mucap} hold. Then \eqref{eq:mucap} and $\mu(\pi(\Gamma_0))=1$ imply that
\begin{equation} \label{eq:mucap2}
\mu\left(\bigcap_{n=1}^{\infty}\pi(\Gamma_n)\right)\geq 1-\sum_{n=1}^{\infty} \frac{1}{2^{n+1}}=\frac 12>0.
\end{equation}
Let us consider
$$C=\bigcap_{n=1}^{\infty} C_n \quad \textrm{and} \quad K=g^{-1}(C).$$

We show that for all $f\in \iK$
\begin{equation} \label{eq:numz} \iZ(f-g)\cap K\subset \{\textrm{isolated points of } \iZ(f-g)\}.
\end{equation}
For each $n\in \NN^+$ and $i\in \{1,\dots,2^n\}$ define the open interval
$$ I_{i,n}=\left(\frac{i-1}{2^n}, \frac{i}{2^n}\right),$$
first we prove that
\begin{equation} \label{eq:diamg} \diam g^{-1}(I_{i,n}) = 2\sum_{k=n+1}^{\infty} \alpha_k\leq 4\alpha_{n+1}.
\end{equation}
The inequality is clearly implied by \eqref{eq:alfa}, so we only need to check the equality. Note that it is important that the $I_{i,n}$ are open,
otherwise the constant pieces of the graph would make the pre-image much bigger.
If $y = \frac{2i-1}{2^{n+1}}$ is the midpoint of $I_{i,n}$ then it is easy to see that $g^{-1} (y) = \{ z_{k_1\dots k_n} \}$ for some $(k_1,\dots, k_n)\in \{-1,1\}^n$. Then $I_{i,n} = U (y, \sum_{k=n+1}^\infty \frac{1}{2^{k+1}})$, and one can check using the definition of $g$ that this corresponds to $g^{-1}(I_{i,n}) = U (z_{k_1\dots k_n}, \sum_{k=n+1}^\infty \alpha_{k})$, which proves \eqref{eq:diamg}.

Let $x\in \iZ(f-g)\cap K$, it is enough to show that if $z\in [0,1]$ with $0<|x-z|\leq 4\alpha_1$ then $f(z)\neq g(z)$. As $\alpha_n\searrow 0$,
there exists a unique number $n\in \NN^+$ such that
\begin{equation} \label{eq:xz<} 4\alpha_{n+1}< |x-z|\leq  4\alpha_{n}.\end{equation}
Since $x\in K$ implies that $g(x)\in C_n$, there exists an $i\in \{1,\dots,2^n\}$ such that
\begin{equation} \label{eq:distg} g(x)\in I_{i,n} \textrm{ and } \dist(\{g(x)\},\partial I_{i,n})\geq h(4\alpha_n).
\end{equation}
Then \eqref{eq:diamg} and \eqref{eq:xz<} imply that $|x-z|>\diam g^{-1}(I_{i,n})$, so $x\in g^{-1}(I_{i,n})$ yields that
$z\notin  g^{-1}(I_{i,n})$. Therefore $g(z)\notin I_{i,n}$ and \eqref{eq:distg} imply that
\begin{equation}\label{eq:gxgz} |g(x)-g(z)|\geq \dist(\{g(x)\},\partial I_{i,n})\geq h(4\alpha_n). \end{equation}
The monotonicity of $h$ and \eqref{eq:xz<} yield that
\begin{equation} \label{eq:hxz}  h(|x-z|)\leq h(4\alpha_n). \end{equation}
Therefore $f(x)=g(x)$, the triangle inequality, \eqref{eq:gxgz}, \eqref{eq:fh}
and \eqref{eq:hxz} imply that
\begin{align*} |f(z)-g(z)|&=|(g(x)-g(z))-(f(x)-f(z))| \\
&\geq |g(x)-g(z)|-|f(x)-f(z)| \\
&>h(4\alpha_n)-h(|x-z|)\geq 0,
\end{align*}
thus $f(z)\neq g(z)$. Hence \eqref{eq:numz} holds.

By \eqref{eq:mucap2} and \eqref{eq:numz} it is enough to show that
\begin{equation} \label{eq:gamfg} \bigcap_{n=1}^{\infty}\pi(\Gamma_n)\subset \{f\in \iK: \iZ(f-g)\cap K\neq \emptyset\}.
\end{equation}
Let us fix $f\in \bigcap_{n=1}^{\infty}\pi(\Gamma_n)$, we need to find an $x\in K$ such that $f(x)=g(x)$.
For every $n\in \NN^+$ we can select a point $(x_n,f(x_n))\in \Gamma_n$.
We can choose a convergent subsequence $\lim_{k\to \infty} x_{n_k}=x$ for some $x\in [0,1]$.
Then \eqref{eq:gamma} yields that for every $k\in \NN^+$ we have
$x_{n_k}=z_{\phi(i_k,n_k)}$ and $f(x_{n_k})\in I_{i_k,n_k}$ for some $i_k\in \{1,\dots, 2^{n_k}\}$.
The definition of $g$ implies that $g(x_{n_k})$ is the midpoint of
$I_{i_k,n_k}$, so $|f(x_{n_k})-g(x_{n_k})|\leq 2^{-n_k}\leq 2^{-k}$ for all $k\in \NN^+$. Therefore
$$f(x)=\lim_{k\to \infty} f(x_{n_k})=\lim_{k\to \infty} g(x_{n_k})=g(x).$$
Then \eqref{eq:gamma} yields $f(x_{n_k})\in C_{n_k}$ for all $k\in \NN^+$. By \eqref{eq:cn} we have $C_{n+1}\subset C_n$ for all $n\in \NN^+$, so
$f(x)\in \bigcap_{k=1}^{\infty} C_{n_k}=\bigcap_{n=1}^{\infty} C_n=C$.
Thus $x=g^{-1}(f(x))\in g^{-1}(C)=K$, hence \eqref{eq:gamfg} holds.
The proof is complete.
\end{proof}

Now we turn to the question concerning positively many level sets.

\begin{lemma} \label{l:Bnonshy} Let $\Delta$ be a family of subsets of $[0,1]$ and consider
\begin{align*} \iA&=\{f\in C[0,1]: f^{-1}(0)\in \Delta\}, \\
\iB&=\{f\in C[0,1]: \exists^{\lambda} y\in \RR \textrm{ such that } f^{-1}(y)\in \Delta\}.
\end{align*}
If $\iA$ is a non-shy Borel set then $\iB$ is a non-shy Borel set, too.
\end{lemma}

\begin{proof} Let $\iL$ be the one-dimensional Lebesgue measure defined on the constant functions of $C[0,1]$, then
clearly
\begin{equation} \label{eq:BB} \iB=\{f\in C[0,1]: \iL(\iA-f)>0\}.
\end{equation}
As $\iA$ is a Borel set, the function $f \mapsto \iL(\iA-f)$ is Borel measurable by \cite[417A]{Fr},
so $\iB$ is Borel. Assume to the contrary that $\iB$ is shy. Then there is a Borel probability measure $\mu$ on $C[0,1]$
such that $\mu(\iB-g)=0$ for all $g\in C[0,1]$, so \eqref{eq:BB} yields
$$\mu(\{f\in C[0,1]: \iL(\iA+g-f)>0\})=\mu(\iB-g)=0.$$
Therefore we obtain by Fubini's theorem that for all $g\in C[0,1]$
$$(\mu \ast \iL) (\iA+g)=\int_{C[0,1]} \iL(\iA+g-f) \, \mathrm{d} \mu(f)=0.$$
Although $\mu \ast \iL$ is only $\sigma$-finite,
by restricting and normalizing it we obtain a probability measure witnessing the shyness of $\iA$.
This is a contradiction, thus the proof is complete.
\end{proof}

\begin{theorem} \label{t:nonshy}
The set
$$\iB=\{f\in C[0,1]:  \exists^{\lambda} y\in \RR \textrm{ such that } f^{-1}(y) \textrm{ is a singleton}\}$$
is non-shy in $C[0,1]$.
\end{theorem}

\begin{proof} Lemma~\ref{l:Bor} and Theorem~\ref{t:nonshy} yield that $\iA=\{f\in C[0,1]: \#f^{-1}(0)=1\}$
is a non-shy Borel set. Lemma~\ref{l:Bnonshy} with $\Delta=\{ \{x\} : x \in [0,1]\}$
implies that $\iB$ is also a non-shy Borel set.
\end{proof}

\begin{remark} Similar arguments yield that for all $n\in \NN^+$ the sets
\begin{align*} \iA_n&=\{f\in C[0,1]: \#f^{-1}(0)=n\}, \\
\iB_n&=\{f\in C[0,1]:  \exists^{\lambda} y\in \RR \textrm{ with } \#f^{-1}(y)=n\}
\end{align*}
are non-shy, the details are left to the reader. The sets $\iA_n$ are pairwise disjoint.
\end{remark}

\section{All non-extremal level sets can be of maximal dimension} \label{s:maxlevel}

For the prevalent $f\in C[0,1]$ the sets $f^{-1}(\min f)$ and $f^{-1}(\max f)$ are singletons, see e.g. \cite{BDE}. The aim of this section is to prove that all other non-empty level sets can be of Hausdorff dimension one. This is a complementary result to Theorem~\ref{t:nonshy}.

\begin{theorem}  \label{t:maxlevel} The set
$$\iC=\{f\in C[0,1]: \dim_H f^{-1}(y)=1 \textrm{ for all } y\in  (\min f, \max f)\}$$
is non-shy in $C[0,1]$.
\end{theorem}

\begin{proof} Let $\iK\subset C[0,1]$ be an arbitrarily fixed compact set. By Lemma~\ref{l:nshy} it is enough to construct a $g\in C[0,1]$ such that
$\iK+g\subset \iC$. First we construct $g$. By Lemma~\ref{l:iK} there is a strictly increasing function $h\in C[0,1]$
such that $h(0)=0$ and for all $f\in \iK$ and $x,z\in [0,1]$ we have
$$|f(x)-f(z)|\leq h(|x-z|).$$
For all $n\in \NN^+$ fix positive integers $a_n>b_n$ such that
\begin{equation} \label{eq:aibi} a_n\geq \max \left\{\frac{1}{h^{-1}(2^{-(n+2)})},2^{5n^2}\right\} \quad \textrm{and} \quad b_n=\left\lceil\frac{a_n}{32}\right\rceil,
\end{equation}
where $\lceil \cdot \rceil$ denotes the upper integer part. For all $n\in \NN^+$ let
$$p_n=\frac{1}{a_1\cdots a_n}.$$
Define $g_n\in C[0,1]$ for all $n\in \NN^+$ as
\begin{equation*} g_n(x)= \begin{cases}
(-1)^i2^{-n} & \textrm{ if } x=ip_n \textrm{ and } 0\leq i\leq a_1 \cdots a_n,  \\
\textrm{affine} & \textrm{ on } [(i-1)p_n,ip_n] \textrm{ for all } 1\leq i\leq a_1 \cdots a_n.
\end{cases}
\end{equation*}
Let us define $g\in C[0,1]$ and $G_n\in C[0,1]$ for all $n\in \NN^+$ as
$$g=\sum_{i=1}^{\infty} g_i \quad \textrm{and} \quad G_n=\sum_{i=1}^{n} g_i.$$

Now we prove that $\iK+g\subset \iC$.
Let us fix $f\in \iK$ and $y\in (\min (f+g),\max (f+g))$, we need to prove that $\dim_H (f+g)^{-1}(y)=1$. As $f+G_n$ uniformly
converges to $f+g$, the intermediate value theorem implies that there is an $m\in \NN^+$ and $x_{\emptyset}\in [0,1]$ such that $(f+G_m)(x_{\emptyset})=y$.
For all $n\in \NN$ let
\begin{align*} \iI_n&=\prod_{i=1}^{n} \{1,\dots, b_{m+i}\} \quad \textrm{and} \\
\iJ_{n}&=\{[(i-1)p_{m+n},ip_{m+n}]: 1\leq i\leq a_1\cdots a_{m+n}\},
\end{align*}
where we use the convention $\iI_0=\{\emptyset\}$ and $(i_1,\dots,i_0)=\emptyset$. Let $I_{\emptyset}\in \iJ_0$ such that $x_{\emptyset}\in I_{\emptyset}$. We construct for each $n\in \NN$ and $(i_1,\dots,i_n)\in \iI_n$ an interval $I_{i_1\dots i_n}\in \iJ_n$ and a point $x_{i_1\dots i_n}\in I_{i_1\dots i_n}$ such that for all $(i_1,\dots, i_{n+1})\in \iI_{n+1}$ we have
\begin{enumerate}
\item \label{eq22} $I_{i_1\dots i_{n+1}}\subset I_{i_1\dots i_n}$,
\item \label{eq11} $(f+G_{m+n})(x_{i_1\dots i_n})=y$.
\end{enumerate}
By definition \eqref{eq11} holds for $x_{\emptyset}$ and $I_{\emptyset}$.
Assume by induction that for some fixed $n\in \NN$ for each $(i_1,\dots,i_n)\in \iI_n$ the interval $I_{i_1\dots i_n}$ and the point $x_{i_1\dots i_n}\in I_{i_1\dots i_n}$ are defined. Let us fix $(i_1,\dots,i_n)\in \iI_n$. We can choose $b_{m+n+1}$ distinct elements of $\iJ_{n+1}$ which are subsets of  $I_{i_1\dots i_n}\cap U(x_{i_1\dots i_n}, p_{m+n}/16)$,
let us enumerate them as $I_{i_1\dots i_{n+1}}$ $(1\leq i_{n+1}\leq b_{m+n+1})$, then \eqref{eq22} holds. Fix $i_{n+1}\in \{1,\dots, b_{m+n+1}\}$ and define $u_{i_1\dots i_{n+1}},v_{i_1\dots i_{n+1}}\in I_{i_1\dots,i_{n+1}}$ such that
$$g_{m+n+1}(u_{i_1\dots i_{n+1}})=-2^{-(m+n+1)} \quad \textrm{and} \quad g_{m+n+1}(v_{i_1\dots i_{n+1}})=2^{-(m+n+1)}.$$
It is enough to prove that
\begin{equation} \label{eq:yeq} (f+G_{m+n+1})(u_{i_1\dots i_{n+1}})\leq y \leq (f+G_{m+n+1})(v_{i_1\dots i_{n+1}}),
\end{equation}
then by the intermediate value theorem we can choose an $x_{i_1\dots i_{n+1}}\in I_{i_1\dots i_{n+1}}$ satisfying \eqref{eq11}.
We prove the second inequality of \eqref{eq:yeq} only, the proof of the first one is analogous. As $(f+G_{m+n})(x_{i_1\dots i_n})=y$ and $g_{m+n+1}(v_{i_1\dots i_{n+1}})=2^{-(m+n+1)}$, it is enough to prove that
\begin{equation} \label{eq:gm} |(f+G_{m+n})(x_{i_1\dots i_n})-(f+G_{m+n})(v_{i_1\dots i_{n+1}})|\leq 2^{-(m+n+1)}.
\end{equation}
Since $x_{i_1\dots i_n},v_{i_1\dots i_{n+1}}\in I_{i_1\dots i_{n}}$, the definition of $h$, $p_{m+n}$ and $a_{m+n}$ imply that
\begin{equation} \label{eq:fxi} |f(x_{i_1\dots i_n})-f(v_{i_1\dots i_{n+1}})|\leq h(p_{m+n})\leq h(1/a_{m+n})\leq 2^{-(m+n+2)}.
\end{equation}
It is easy to show that for all $i\in \NN^+$ the function $g_i$ is Lipschitz and $\Lip(g_i)=2^{1-i}p^{-1}_i=2^{1-i}a_1\cdots a_i$. Thus
$G_{m+n}$ is Lipschitz and $a_{m+n}\geq 2^{m+n}$ implies that
\begin{align*} \Lip(G_{m+n})&\leq \sum_{i=1}^{m+n} \Lip(g_i)=\sum_{i=1}^{m+n} 2^{1-i}a_1\cdots a_i \\
&\leq 2a_1\cdots a_{m+n-1}+2^{1-(m+n)}a_1\cdots a_{m+n} \\
&\leq 2^{2-(m+n)}a_1\cdots a_{m+n}=\frac{2^{2-(m+n)}}{p_{m+n}}.
\end{align*}
Therefore $I_{i_1\dots i_{n}} \subset U(x_{i_1\dots i_n}, p_{m+n}/16)$ yields that
\begin{align} \label{eq:gm2}
\begin{split}
|G_{m+n}(x_{i_1\dots i_n})-G_{m+n}(v_{i_1\dots i_{n+1}})|&\leq \Lip(G_{m+n})|x_{i_1\dots i_n}-v_{i_1\dots i_{n+1}}| \\
&\leq \frac{2^{2-(m+n)}}{p_{m+n}} \cdot \frac{p_{m+n}}{16}=2^{-(m+n+2)}.
\end{split}
\end{align}
Equations \eqref{eq:fxi} and \eqref{eq:gm2} imply \eqref{eq:gm}, and the induction is complete.
For all $n\in \NN^+$ let $c_n=a_{m+n}$ and $d_n=b_{m+n}$. Set
$$C=\bigcap_{n=1}^{\infty} \left( \bigcup_{i_1=1}^{d_1} \dots \bigcup_{i_n=1}^{d_n} I_{i_1\cdots i_n}\right).$$
Then $C$ is a $(c_n,d_n)$-type compact set, see Definition~\ref{d:Cantor2}.
Then \eqref{eq:aibi} implies that $a_i\geq 2^{5i^2}$ and $32\geq a_i/b_i$, so for all $n\in \NN^+$ we have
$$c_n\geq 2^{5(m+n)^2}\geq 2^{5(n+1)^2}\geq  \left(\frac{c_1\cdots c_{n+1}}{d_1\cdots d_{n+1}}\right)^{n+1}.$$
Therefore Lemma~\ref{l:Can} implies that $\dim_H C=1$.

Finally, in order to prove $\dim_H (f+g)^{-1}(y)=1$, it is enough to show that $C\subset (f+g)^{-1}(y)$. Let us fix $x\in C$, we prove that $(f+g)(x)=y$.
For all $n\in \NN^+$ pick indices $i_n\in \{1,\dots, d_n\}$ such that $x\in I_{i_1\dots i_n}$,
then clearly $\lim_{n\to \infty} x_{i_1\dots i_n}=x$. As $f+G_{m+n}$ converges uniformly to $f+g$, property \eqref{eq11} implies that
$$(f+g)(x)=\lim_{n\to \infty} (f+G_{m+n})(x_{i_1\dots i_{n}})=y,$$
and the proof is complete.
\end{proof}

Theorems~\ref{t:nonshy} and \ref{t:maxlevel} yield the following.

\begin{corollary} The sets
\begin{align*} \iB&=\{f\in C[0,1]:  \exists^{\lambda} y\in \RR \textrm{ such that } f^{-1}(y) \textrm{ is a singleton}\}, \\
\iC&=\{f\in C[0,1]: \dim_H f^{-1}(y)=1 \textrm{ for all } y\in  (\min f, \max f)\}
\end{align*}
are disjoint non-shy sets in $C[0,1]$, so they are neither shy nor prevalent.
\end{corollary}

\section{Dimensions of graphs of prevalent continuous maps} \label{s:graph}

By product of two metric spaces $(X,d_X)$ and $(Y,d_Y)$ we will always mean the $l^2$-product, that is,
$$d_{X\times Y}((x_{1},y_{1}),(x_{2},y_{2}))= \sqrt{d^{2}_{X}(x_1,x_2)+d^{2}_{Y}(y_1,y_2)}.$$
For $E\subset X\times Y$ and $y\in Y$ let $E^y= \{ x \in X : (x,y) \in E\}$.

The following lemma is basically \cite[Theorem~7.7]{Ma}.
It is only stated there in the special case $X=A \subset \mathbb{R}^m$, but the proof works verbatim for all metric spaces $X$.

\begin{lemma} \label{lip}
Let $X$ be a metric space and let $d\in \NN^+$. If $f\colon X \to \mathbb{R}^{d}$ is Lipschitz and $t\geq d$ then
 \begin{equation*} \int_{\mathbb{R}^d} ^{\star} \mathcal{H}^{t-d}(f^{-1}(y))\, \mathrm{d} \lambda^{d}(y)\leq c(d)
 \Lip(f)^{d} \mathcal{H}^{t}(X), \end{equation*}
where $\int ^{\star}$ denotes the upper integral and $c(d)$ is a
finite constant depending on $d$ only.
\end{lemma}

Let $E\subset X\times \RR^d$ and define $f\colon E\to \RR^d$ as $f(x,y)=y$. Applying Lemma~\ref{lip} for
$f$ yields the following lemma.

\begin{lemma} \label{l:Hsec} Let $X$ be a metric space and let $d\in \NN^+$. If $E\subset X\times \RR^d$ then
for $\lambda^d$-almost every $y\in \RR^d$ we have
$$\dim_H (E^y)\leq \max\{0,\dim_H E-d\}.$$
\end{lemma}

Recently Orponen \cite[Cor.~1.2]{O} has shown that Lemma~\ref{lip} does not remain true if we replace Hausdorff measures by packing measures. The analogous version of Lemma~\ref{l:Hsec} holds, see the
proof of \cite[Lemma~5]{F1} with the natural modifications.

\begin{lemma} \label{l:Psec} Let $X$ be a metric space and let $d\in \NN^+$. If $E\subset X\times \RR^d$ then
for $\lambda^d$-almost every $y\in \RR^d$ we have
$$\dim_P (E^y)\leq \max\{0,\dim_P E-d\}.$$
\end{lemma}

For the following lemma see \cite[Theorem~8.10]{Ma}. It is only stated there for subsets of Euclidean spaces,
but the same proof works here as well.

\begin{lemma} \label{l:box}
If $X,Y$ are non-empty metric spaces then
\begin{align*} \dim_{H}(X\times Y) &\leq \dim_{H}X+\dim_{P} Y, \\
\dim_P(X\times Y)&\leq \dim_P X+\dim_P Y.
\end{align*}
\end{lemma}

\begin{theorem} \label{t:graph} Let $K$ be an uncountable compact metric space and let $d\in \NN^+$.
Then for the prevalent $f\in C(K,\RR^d)$ we have
$$\dim_H \graph(f)=\dim_H K+d.$$
\end{theorem}

\begin{proof} Lemma~\ref{l:box} and $\dim_{P} \RR^d=d$ yield that for all $f\in C(K,\RR^d)$
$$\dim_H \graph(f)\leq \dim_H (K\times \RR^d)\leq \dim_H K+d,$$
so it is enough to prove the opposite inequality for the prevalent $f$.

If $\dim_H K=0$ then Theorem~\ref{t:D} yields that for the prevalent $f\in C(K,\RR^d)$ we have $\inter f(K)\neq \emptyset$,
so $\dim_H f(K)=d$. As $f(K)$ is a Lipschitz image of $\graph(f)$ and Hausdorff dimension cannot increase under a Lipschitz map,
we obtain
$$\dim_H \graph(f)\geq \dim_H f(K)=d=\dim_H K+d,$$
and we are done. Hence we may assume that $\dim_H K>0$. Consider
\begin{align*} \iA=\{&f\in C(K,\mathbb{R}^d): \textrm{for all } s<\dim_H K \textrm{ there exists a non-empty} \\
&\textrm{open set } U_{f,s}\subset \mathbb{R}^d \textrm{ such that } \dim_H f^{-1}(y)>s \textrm{ for all } y\in U_{f,s}\}.
\end{align*}
Let $f\in \iA$ and $s\in (0,\dim_H K)$ be arbitrarily given. Since Theorem~\ref{t:main} yields that $\iA$ is prevalent in $C(K,\mathbb{R}^d)$,
it is enough to show that $\dim_H \graph (f)\geq s+d$. Let $E=\graph(f)\subset K\times \RR^d$, then for all $y\in U_{f,s}$ we have $\dim_H E^y=\dim_H f^{-1}(y)\geq s$. As $\lambda^d(U_{f,s})>0$ and $s>0$, Lemma~\ref{l:Hsec} implies that $\dim_H E\geq s+d$.
The proof is complete.
\end{proof}

\begin{theorem} \label{t:graph2} Let $K$ be an uncountable, non-exploding compact metric space and let $d\in \NN^+$.
Then for the prevalent $f\in C(K,\RR^d)$ we have
$$\dim_P \graph(f)=\dim_P K+d.$$
\end{theorem}

\begin{proof} We can repeat the proof of Theorem~\ref{t:graph}, only replace Hausdorff dimension with packing dimension, and apply Corollary~\ref{c:Mainp} and Lemma~\ref{l:Psec} instead of Theorem~\ref{t:main} and Lemma~\ref{l:Hsec}, respectively.
\end{proof}

\section{Finer results with generalized Hausdorff measures} \label{s:gauge}

In this section we indicate how to obtain sharper versions of the main results.
Since the proofs were quite technical already, we decided not to include these stronger forms in
the main body of the paper, only give a brief sketch in this separate section.

A function $\varphi \colon [0,\infty)\to [0,\infty)$ is defined to be a \emph{gauge function} if it is non-decreasing
and $\varphi(0)=0$. For a metric space $X$ let
\begin{align*}
\mathcal{H}^{\varphi}(X)&=\lim_{\delta\to 0+}\mathcal{H}^{\varphi}_{\delta}(X)
\mbox{, where}\\
\mathcal{H}^{\varphi}_{\delta}(X)&=\inf \left\{ \sum_{i=1}^\infty \varphi(\diam
A_{i}): X \subset \bigcup_{i=1}^{\infty} A_{i},~
\forall i \, \diam A_i \le \delta \right\}.
\end{align*}
We call $\mathcal{H}^{\varphi}$ the \emph{$\varphi$-Hausdorff measure}, which extends the concept of
classical Hausdorff measure. There are examples when this finer notion of
measure is needed, this is the case when we want
to measure the level sets of the linear Brownian motion or the range of a $d$-dimensional Brownian motion.
For more information see \cite{MP} and \cite{Ro}.

Let $\iG$ be the set of gauge functions and for all $s>0$ let
$$\iG(s)=\left\{\varphi \in \iG: \lim_{r\to 0+} \frac{\varphi(r)}{r^{s}}=\infty\right\}.$$
Now we show how to generalize Theorem~\ref{t:real}, Theorem~\ref{t:maxlevel} and Theorem~\ref{t:FH}.
First we need to extend Lemma~\ref{l:Can}.

\begin{lemma} \label{l:Can2} Let $\varphi\in \iG(1)$ be a gauge function. Let us define the non-decreasing function
$\Phi\colon [1,\infty)\to [1,\infty)$ as
\begin{equation} \label{eq:psi} \Phi(x)=\sup\{r\in \RR^+: r\varphi(1/r)\leq x\}+1,
\end{equation}
where $\sup \emptyset =0$ by convention.
Let $C\subset \RR$ be an $(a_n,b_n)$-type compact set such that for all $n\in \NN^+$
\begin{equation*} a_n\geq \Phi\left(\frac{a_1\cdots a_{n+1}}{b_1\cdots b_{n+1}}\right).
\end{equation*}
Then $\iH^{\varphi}(C)>0$.
\end{lemma}

\begin{proof} Let $\mu$ be the same measure as in the proof of Lemma~\ref{l:Can}, then similar arguments yield that
all Borel sets $B\subset C$ with $\diam B\leq 1$ satisfy
$$\mu(B)\leq 4 \varphi(\diam B).$$
Therefore the mass distribution principle for generalized Hausdorff measures implies that $\iH^{\varphi}(C)>0$,
see also \cite[Proposition~6.44 (i)]{MP}.
\end{proof}

Instead of Theorem~\ref{t:real} we can prove the following stronger form.

\begin{theorem} \label{t:real2} Let $K\subset \RR$ be a compact set with $\lambda(K)>0$ and let $d\in \NN^+$.
Let $\varphi \in \iG(1)$ be a gauge function. Then for the prevalent $f\in C(K,\RR^d)$ there exists an open set $U_{f}\subset \RR^d$ such that $\lambda(f^{-1}(U_f))=\lambda(K)$ and for all $y\in U_{f}$ we have
$$\iH^{\varphi} (f^{-1} (y))>0.$$
\end{theorem}

\begin{proof}
Let $\Phi\colon [1,\infty)\to [1,\infty)$ be the function defined in \eqref{eq:psi}.
Let us follow the proof of Theorem~\ref{t:real}, the only difference is that we define the numbers $a_n$
by induction such that $b_n=(2s)^{-(n+3)}a_n$ are integers and for all $n\in \NN^+$ we have
\begin{equation*} a_n\geq \max\left\{(2s)^{8n}(a_1\cdots a_{n-1}),  \Phi\left(\frac{a_1\cdots a_{n+1}}{b_1\cdots b_{n+1}}\right) \right\}.
\end{equation*}
Then applying Lemma~\ref{l:Can2} instead of Lemma~\ref{l:Can} concludes the proof.
\end{proof}

Instead of Theorem~\ref{t:maxlevel} we can prove the following stronger form.

\begin{theorem}  \label{t:maxlevel2} Let $\varphi \in \iG(1)$ be a gauge function. Then the set
$$\iC=\{f\in C[0,1]: \iH^{\varphi}(f^{-1}(y))>0 \textrm{ for all } y\in  (\min f, \max f)\}$$
is non-shy in $C[0,1]$.
\end{theorem}

\begin{proof}
Let $\Phi\colon [1,\infty)\to [1,\infty)$ be the function defined in \eqref{eq:psi}.
Let us follow the proof of Theorem~\ref{t:maxlevel}, the only difference is that in \eqref{eq:aibi}
we replace $2^{5n^2}$ by $\Phi(2^{5n})$ and we apply Lemma~\ref{l:Can2} instead of Lemma~\ref{l:Can}.
\end{proof}

Fraser and Hyde proved in \cite{FH} that the prevalent $C[0,1]$ has graph of Hausdorff dimension $2$. They observed that
$\iH^{2}(\graph(f))=0$ for all $f\in C[0,1]$ by Fubini's theorem,
and raised the problem what we can say using different gauge functions. The following theorem solves this problem
by stating that the graph of the prevalent $f\in C[0,1]$ is as large as possible according to this finer scale, too.

\begin{theorem} \label{t:graphg} Let $d\in \NN^+$ and let $\psi \in \iG(d+1)$ be a gauge function. Then for the
prevalent $f\in C([0,1],\RR^d)$ we have
$$\iH^{\psi} (\graph(f))>0.$$
\end{theorem}

Before sketching the proof of Theorem~\ref{t:graphg} we need two lemmas.

\begin{lemma} \label{l:gauge} Let $d\in \NN^+$ and let $\psi\in \iG(d+1)$ be a gauge function. Then there is a gauge function
$\varphi\in \iG(1)$ such that for all $r\in [0,1]$ we have
$$\varphi(r)r^d\leq \psi(r).$$
\end{lemma}

\begin{proof} Let $\varphi(0)=0$ and $\varphi(r)=\psi(1)$ for all $r>1$. Define $\varphi(r)=\inf_{s\in [r,1]} \psi(s) s^{-d}$ if $0<r\leq 1$.
Then clearly $\varphi(r)r^d\leq \psi(r)$ for all $r\in [0,1]$, and it is easy to check that $\varphi$ is a gauge function with
$\varphi\in \iG(1)$.
\end{proof}

For the following lemma see the proof of \cite[Theorem~7.7]{Ma} with the natural modifications.

\begin{lemma} \label{lip2}
Let $X$ be a metric space and let $d\in \NN^+$. Let $\varphi,\sigma$ be gauge functions such that
$\sigma(r)=\varphi(r)r^d$ for all $r\geq 0$. If $g\colon X \to \mathbb{R}^{d}$ is Lipschitz then
 \begin{equation*} \int_{\mathbb{R}^d} ^{\star} \mathcal{H}^{\varphi}(g^{-1}(y))\, \mathrm{d} \lambda^{d}(y)\leq c(d)
 \Lip(g)^{d} \mathcal{H}^{\sigma}(X), \end{equation*}
where $\int ^{\star}$ denotes the upper integral and $c(d)$ is a
finite constant depending on $d$ only.
\end{lemma}

\begin{proof}[Proof of Theorem~\ref{t:graphg}] By Lemma~\ref{l:gauge} there is a gauge function
$\varphi\in \iG(1)$ such that $\varphi(r)r^d\leq \psi(r)$ for all $r\in [0,1]$. Let us define
$\sigma(r)=\varphi(r)r^d$ for all $r\geq 0$. Consider
\begin{align*} \iA=\{&f\in C([0,1],\mathbb{R}^d): \textrm{ there exists a non-empty open set} \\
&U_{f}\subset \mathbb{R}^d \textrm{ such that } \iH^{\varphi} (f^{-1}(y))>0 \textrm{ for all } y\in U_{f}\}.
\end{align*}
Theorem~\ref{t:real2} yields that $\iA$ is prevalent in $C(K,\mathbb{R}^d)$. Let us fix $f\in \iA$, it is enough to prove that
$\iH^{\psi}(\graph(f))>0$. Let $g\colon [0,1]\times \RR^d\to \RR^d$, $g(x,y)=y$ be the natural projection onto $\RR^d$ and let $X=\graph(f)$.
Applying Lemma~\ref{lip2} for $X$ and $g|_{X}$ implies that $\iH^{\sigma}(\graph(f))>0$. Since $\sigma(r)\leq \psi(r)$ for all $r\in [0,1]$,
we obtain that $\iH^{\psi}(\graph(f))>0$. The proof is complete.
\end{proof}

\section{Open problems} \label{s:open}

Let $\{B(t): t\in [0,1]\}$ be a standard linear Brownian motion.
Antunovi\'c et al.\ \cite[Theorem~1.5]{ABPR} proved that for every $f\in C[0,1]$ the zero set
$\iZ(B-f)$ has Hausdorff dimension at least $1/2$ with positive probability. Moreover, their proof gives that
$\iH^h(\iZ(B-f))>0$, where $h$ is a gauge function such that $h(2^{-n})=2^{-(\beta_1+\dots +\beta_n)}$ and $\beta_n\nearrow 1/2$.
Fubini's theorem implies that, with positive probability,
we have $\iH^h((B-f)^{-1}(y))>0$ for positively many $y$. Peres and Sousi proved a general 0-1 law \cite[Theorem~2.1]{PS1},
which yields that the above property holds with probability one.\footnote{More precisely, for every
closed set $A\subset [0,1]$ define the random variable $\Psi(A)$ such that $\Psi(A)=1$ if $\iH^h(A\cap (B-f)^{-1}(y))>0$ for positively many $y$ and $\Psi(A)=0$ otherwise. Applying \cite[Theorem~2.1]{PS1} for $\Psi$ yields that $\mathbb{P}(\Psi([0,1])>0)\in \{0,1\}$.} Therefore, almost surely,
$\dim_H (B-f)^{-1}(y)\geq 1/2$ for positively many $y$. We would like to know whether `positively many' can be replaced by
`non-empty open' and `almost every with respect to the occupation measure'.
The following problem asks whether the Wiener measure witnesses a weaker form of Corollary~\ref{c:occup}.

\begin{problem} Let $\{B(t): t\in [0,1]\}$ be a standard one-dimensional Brownian motion and let $f\in C[0,1]$.
Does there exist a random non-empty open set $U\subset \RR$ such that, almost surely, for all $y\in U$ we have
$$\dim_H (B-f)^{-1}(y)\geq 1/2?$$
Let $U$ be the maximal such open set. Does $\lambda((B-f)^{-1}(U))=1$ hold almost surely?
\end{problem}

Let $0<\alpha<1$ and let $C^{\alpha}[0,1]$ denote the set of $\alpha$-H\"older continuous functions $f\colon [0,1] \to \RR$ endowed with the norm
$$||f||_{\alpha}=\sup_{x\in [0,1]} |f(x)|+\sup_{0\leq x<y\leq 1} \frac{|f(x)-f(y)|}{|x-y|^{\alpha}}.$$
Clearly $C^{\alpha}[0,1]$ is a Banach space. Clausel and Nikolay \cite[Theorem~2]{CN}
proved that the graph of the prevalent $f\in C^{\alpha}[0,1]$ is of Hausdorff dimension $2-\alpha$, see also \cite{BH} for a generalization.
Studying the level sets seems to be a more delicate matter.

\begin{problem} Let $0<\alpha<1$. Is it true that for the prevalent $f\in C^{\alpha}[0,1]$ there exists an open set
$U_f\subset \RR$ such that $\lambda(f^{-1}(U_f))=1$ and for all $y\in U_f$ we have
\begin{equation*} \label{eq:Hold} \dim_H f^{-1}(y)\geq 1-\alpha?
\end{equation*}
Does $\dim_H f^{-1}(y)\geq 1-\alpha$ hold at least for positively many $y$?
\end{problem}

\begin{problem} Can we omit the condition that $K$ is non-exploding
from the Main Theorem, or more generally, from Theorem~\ref{t:Z}?
\end{problem}

We would like to describe the compact metric spaces $K$ for which Theorem~\ref{t:main} can be strengthened.
Here we consider only the one-dimensional case.

\begin{problem} Characterize the compact sets $K\subset \RR$ such that for the prevalent
$f\in C(K,\RR)$ there is a non-empty open set $U_f\subset \RR$ such that for all
$y\in U_f$ we have $\dim_H f^{-1}(y)=\dim_H K$.
\end{problem}

\begin{problem} Characterize the compact sets $K\subset \RR$ such that for the prevalent
$f\in C(K,\RR)$ there exists a $y_f\in \RR$ such that $\dim_H f^{-1}(y_f)=\dim_H K$.
\end{problem}

\subsection*{Acknowledgments}
We are indebted to Y. Peres, M. Vizer and O. Zindulka for numerous illuminating conversations.


\begin{thebibliography}{99}

\bibitem{ABPR} T. Antunovi\'c, K. Burdzy, Y. Peres, J. Ruscher,
Isolated zeros for Brownian motion with variable drift, \textit{Electron. J. Probab.} \textbf{16} (2011), no. 65, 1793--1814.

\bibitem{B} R. Balka, Inductive topological Hausdorff dimensions and fibers of generic continuous functions,
\textit{Monatsh. Math.} 174 (2014), no. 1, 1--28.

\bibitem{BBE} R. Balka, Z. Buczolich, M. Elekes, A new fractal dimension: The topological Hausdorff
dimension, \textit{Adv. Math.} \textbf{274} (2015), 881--927.

\bibitem{BBE2} R. Balka, Z. Buczolich, M. Elekes, Topological Hausdorff dimension and level sets of
generic continuous functions on fractals, \textit{Chaos Solitons Fractals} \textbf{45} (2012), no. 12, 1579--1589.

\bibitem{BDE} R. Balka, U. B. Darji, M. Elekes, Bruckner-Garg-type results with respect to Haar null sets in $C[0,1]$, to appear
in \textit{Proc. Edinb. Math. Soc.}, arXiv:1311.5293.

\bibitem{BH} F. Bayart, Y. Heurteaux,
On the Hausdorff dimension of graphs of prevalent continuous functions on compact sets, In:
\textit{Further Developments in Fractals and Related Fields}, edited by Julien Barral and St\'ephane Seuret, Springer, 2013, 25--34.

\bibitem{Bi} P. Billingsley, \textit{Probability and measure}, Third edition, John Wiley \& Sons, 1995.

\bibitem{CN} M. Clausel, S. Nikolay, Some prevalent results about strongly monoH\"older functions,
Nonlinearity \textbf{23} (2010), no.~9, 2101--2116.

\bibitem{C} J. P. R. Christensen, On sets of Haar measure zero in abelian Polish groups,
\textit{Israel J. Math.} \textbf{13} (1972), 255--260.

\bibitem{D} R. Dougherty, Examples of non-shy sets, \textit{Fund. Math.} \textbf{144} (1994), 73--88.

\bibitem{EV} M. Elekes, Z. Vidny\'anszky, Haar null sets without $G_\delta$ hulls, to appear in \textit{Israel J. Math.}, arXiv:1312.7667.

\bibitem{F} K. Falconer, \textit{Fractal geometry: Mathematical
foundations and applications}, Second Edition, John Wiley \& Sons,
2003.

\bibitem{F1} K. Falconer, Sets with large intersection properties, \textit{J. London Math. Soc.} \textbf{49} (1994), no. 2,
267--280.

\bibitem{F2} K. Falconer, \textit{Techniques in fractal geometry}, John Wiley \& Sons, 1997.

\bibitem{FF} K.~J. Falconer, J.~M. Fraser,
The horizon problem for prevalent surfaces, \textit{Math. Proc. Cambridge Philos. Soc.} \textbf{151} (2011), 355--372.

\bibitem{FH} J. M. Fraser, J. T. Hyde, The Hausdorff dimension of graphs of prevalent continuous
functions, \textit{Real Anal. Exchange} \textbf{37} (2011), no. 2, 333--352.

\bibitem{Fr} D. H. Fremlin, \textit{Measure Theory, vol. 4,
Topological Measure Spaces}, Torres Fremlin, 2003.

\bibitem{GJMNOP} V. Gruslys, J. Jonu\v{s}as, V. Mijovi\`c, O. Ng, L. Olsen, I. Petrykiewicz, Dimensions of
prevalent continuous functions, \textit{Monatsh. Math.} \textbf{166} (2012), 153--180.

\bibitem{Ha} P. R. Halmos, \textit{Measure theory}, Springer-Verlag, 1974.

\bibitem{Ho} J. D. Howroyd, On dimension and on the existence of sets of finite positive Hausdorff measure, \textit{Proc. London Math. Soc.}
\textbf{70} (1995), no. 3, 581--604.

\bibitem{HSY} B. Hunt, T. Sauer, J. Yorke, Prevalence: a translation-invariant ``almost
every'' on infinite-dimensional spaces, \textit{Bull. Amer. Math. Soc.} \textbf{27} (1992), 217--238.

\bibitem{IT} S. Ikeda, M. Tamashiro, Frostman theorem for the packing measure,
\textit{C. R. Math. Acad. Sci. Paris} \textbf{320} (1995), no. 1, 1445--1448.

\bibitem{JP} H. Joyce, D. Preiss, On the existence of subsets of finite positive packing measure,
\textit{Mathematika} \textbf{42} (1995), 15--24.

\bibitem{K} A. S. Kechris, \textit{Classical descriptive set theory}, Springer-Verlag, 1995.

\bibitem{KMZ} T. Keleti, A. M\'ath\'e, O. Zindulka, Hausdorff dimension of metric spaces and Lipschitz maps onto cubes, \textit{Int. Math. Res. Not. IMRN}, \textbf{2014} (2014), no. 2, 289--302.

\bibitem{Ki} B. Kirchheim, Hausdorff measure and level sets of typical continuous mappings
in Euclidean spaces, \textit{Trans. Amer. Math. Soc.} \textbf{347} (1995),
no. 5, 1763--1777.

\bibitem{Ma} P. Mattila, \textit{Geometry of sets and measures in Euclidean spaces}, Cambridge Studies in Advanced Mathematics No.~44,
Cambridge University Press, 1995.

\bibitem{MW} R. D. Mauldin, S. C. Williams, On the Hausdorff dimension of some graphs, \textit{Trans. Amer. Math. Soc.} \textbf{298} (1986), 793--803.

\bibitem{Mc} M. McClure, The prevalent dimension of graphs, \textit{Real Anal. Exchange} \textbf{23} (1997), 241--246.

\bibitem{MN} M. Mendel, A. Naor, Ultrametric subsets with large Hausdorff dimension, \textit{Invent. Math.} \textbf{192} (2013), no. 1, 1--54.

\bibitem{MP} P. M\"orters, Y. Peres, \textit{Brownian Motion}, with an appendix by
Oded Schramm and Wendelin Werner, Cambridge University Press, 2010.

\bibitem{M} J. Mycielski, Some unsolved problems on the prevalence of ergodicity, instability
and algebraic independence, \textit{Ulam Quart.} \textbf{1} (1992), no. 3, 30–-37.

\bibitem{NZ} A. Nekvinda, O. Zindulka, Monotone metric spaces, \textit{Order} \textbf{29} (2012), no. 3, 545--558.

\bibitem{O} T. Orponen, On the packing measure of slices of self-similar sets,
to appear in \textit{J. Fractal Geom.}, arXiv:1309.3896.

\bibitem{PS1} Y. Peres, P. Sousi, Brownian motion with variable drift: 0-1 laws, hitting probabilities
and Hausdorff dimension, \textit{Math. Proc. Cambridge Philos. Soc.}, \textbf{153} (2012), no. 2, 215--234.

\bibitem{PS} Y. Peres, P. Sousi, Dimension of Fractional Brownian motion with variable drift,
to appear in \textit{Probab. Theory Related Fields}, arXiv:1310.7002.

\bibitem{Ro} C. A. Rogers, \textit{Hausdorff measures}, Cambridge University Press, 1970.

\bibitem{R} C. A. Rogers, Sets non-$\sigma$-finite for Hausdorff measures, \textit{Mathematika} \textbf{9} (1962), no. 2, 95--103.

\bibitem{Sh} A. Shaw, Prevalence, \textit{M.Math Dissertation}, University of St. Andrews, 2010.

\bibitem{SS} M. Sion, D. Sjerve, Approximation properties of measures generated by continuous set functions, \textit{Mathematika}
\textbf{9} (1962), no. 2, 145--156.

\bibitem{S} S. Solecki, On Haar null sets, \textit{Fund. Math.} \textbf{149} (1996), 205--210.

\bibitem{THJ} F. Tops{\o}e, J. Hoffmann-J{\o}rgensen, Analytic spaces and their application,
In: \textit{Analytic Sets}, edited by C. A. Rogers et al.\, Academic Press, London, 1980, 317--401.

\bibitem{Za} L. Zaj\'{\i}\v{c}ek, On differentiability properties of typical continuous functions and Haar null sets,
\textit{Proc. Amer. Math. Soc.} \textbf{134} (2005), no.~4, 1143--1151.

\bibitem{Z} O. Zindulka, Mapping Borel sets onto balls and self-similar sets by Lipschitz and nearly
Lipschitz maps, in preparation.

\end{thebibliography}
\end{document}